\documentclass[12pt, a4paper]{amsart}
\usepackage{amsmath,amssymb,amsthm,booktabs,longtable}
\usepackage{comment}
\usepackage{graphicx}
\usepackage[all]{xy}
\theoremstyle{plain}
\makeatletter 

\@addtoreset{table}{section}
\makeatother 
\newtheorem{thm}{Theorem}[section]
\newtheorem{lem}[thm]{Lemma}

\newtheorem{prop}[thm]{Proposition}
\newtheorem{defn}[thm]{Definition}

\newtheorem{rem}[thm]{Remark}

\newcommand{\simrightarrow}{\mathrel{\stackrel{\sim}{\longrightarrow}}}

\newcommand{\bijarrow}{\mathrel{\stackrel{1:1}{\longleftrightarrow}}}

\newcommand{\Int}{\mathop{\mathrm{Int}}\nolimits}

\newcommand{\Map}{\mathop{\mathrm{Map}}\nolimits}

\newcommand{\Hom}{\mathop{\mathrm{Hom}}\nolimits}

\newcommand{\R}{\mathbb{R}}
\newcommand{\C}{\mathbb{C}}
\makeatletter
    
    \@addtoreset{equation}{section}
  \makeatother
\usepackage{enumerate}%
\makeatother
\begin{document}

 \title{Abundance of nilpotent orbits in real semisimple Lie algebras}
\author{Takayuki Okuda}
\address{Department of Mathematics, Graduate School of Science, Hiroshima University 1-3-1 Kagamiyama, Higashi-Hiroshima, 739-8526 Japan}
\email{okudatak@hiroshima-u.ac.jp}
\date{}
\maketitle

\begin{abstract}
We formulate and prove that 
there are ``abundant'' in nilpotent orbits in real semisimple Lie algebras,
in the following sense. 
If $S$ denotes the collection of hyperbolic elements corresponding the weighted Dynkin diagrams coming from nilpotent orbits, 
then $S$ span the maximally expected space, namely, 
the $(-1)$-eigenspace of the longest Weyl group element.
The result is used to 
the study of fundamental groups of non-Riemannian locally symmetric spaces.
\end{abstract}

\setcounter{tocdepth}{1}
\tableofcontents

\section{Main theorem}\label{section:purpose}

Let $\mathfrak{g}$ be a real semisimple Lie algebra,
$\mathfrak{a}$ a maximally split abelian subspace of $\mathfrak{g}$,
and $W$ the Weyl group of $\Sigma(\mathfrak{g},\mathfrak{a})$.
We fix a positive system $\Sigma^+(\mathfrak{g},\mathfrak{a})$, 
denote by $\mathfrak{a}_+ (\subset \mathfrak{a})$ 
the closed Weyl chamber, 
$w_0$ the longest element of $W$,
and set $\mathfrak{a}^{-w_0} := \{ A \in \mathfrak{a} \mid -w_0 \cdot A = A \}$.

We denote by $\Hom(\mathfrak{sl}(2,\R),\mathfrak{g})$ the set of all Lie algebra homomorphisms from $\mathfrak{sl}(2,\R)$ to $\mathfrak{g}$, and put
\begin{align*}
\mathcal{H}^{n}(\mathfrak{g}) 
	&:= \left\{ \rho \begin{pmatrix} 1 & 0 \\ 0 & -1 \end{pmatrix} \in \mathfrak{g}~\middle|~ \rho \in \Hom(\mathfrak{sl}(2,\R), \mathfrak{g}) \right\}, \\ 
\mathcal{H}^n(\mathfrak{a}_+) &:= \mathfrak{a}_+ \cap \mathcal{H}^n(\mathfrak{g}).
\end{align*}

In this paper, we prove the following theorem:
\begin{thm}\label{thm:intro}
$\mathfrak{a}^{-w_0} = \R\text{-}\mathrm{span} (\mathcal{H}^{n}(\mathfrak{a}_+))$.
\end{thm}

In Theorem \ref{thm:intro}, the inclusion
\begin{align}
\mathfrak{a}^{-w_0} \subset \R\text{-}\mathrm{span} (\mathcal{H}^{n}(\mathfrak{a}_+)) \label{eq:nontri1}
\end{align}
is non-trivial and the opposite inclusion is easy (see Section \ref{section:proof} for more details). 

Theorem \ref{thm:intro} means that 
nilpotent orbits are ``abundant'' in $\mathfrak{g}$ in the following sense (see \eqref{eq:nontriv2}).
Let us denote by $\mathcal{N}(\mathfrak{g})$ the set of all nilpotent elements in $\mathfrak{g}$.
Then by results of Jacobson--Morozov and Kostant, we have 
\[
\mathcal{N}(\mathfrak{g}) 
= \left \{ \rho \begin{pmatrix} 0 & 1 \\ 0 & 0 \end{pmatrix} \in \mathfrak{g} ~\middle|~ \rho \in \Hom(\mathfrak{sl}(2,\R),\mathfrak{g}) \right \} 
\]
and for each nilpotent (adjoint) orbit $\mathcal{O} \in \mathcal{N}(\mathfrak{g})/{\Int \mathfrak{g}}$, 
there exists a unique $A_\mathcal{O} \in \mathcal{H}^n(\mathfrak{a}_+)$ such that $\rho \begin{pmatrix}
1 & 0 \\ 0 & -1
\end{pmatrix} = A_\mathcal{O}$ and $\rho \begin{pmatrix}
0 & 1 \\ 0 & 0
\end{pmatrix} \in \mathcal{O}$ for some $\rho \in \Hom(\mathfrak{sl}(2,\R),\mathfrak{g})$ (see \cite[Section 9.2]{Collingwood-McGovern93} for more details).
Then the correspondence $\mathcal{O} \mapsto A_\mathcal{O}$ gives a surjective map:
\[
\phi: 
\mathcal{N}(\mathfrak{g})/{\Int \mathfrak{g}} \twoheadrightarrow \mathcal{H}^n(\mathfrak{a}_+).
\]
The non-trivial part \eqref{eq:nontri1} of Theorem \ref{thm:intro} is equivalent to the following inclusion:
\begin{align}
\mathfrak{a}^{-w_0} \subset \R\text{-}\mathrm{span}~\phi(\mathcal{N}(\mathfrak{g})/{\Int \mathfrak{g}}). \label{eq:nontriv2}
\end{align}

Theorem \ref{thm:intro} was announced in \cite[Proposition 4.8 (i)]{Okuda13cls}
 which played a key role in proving 
one of the main results in \cite{Okuda13cls} was that 
semisimple symmetric spaces $G/H$ admit properly discontinuous groups 
which are not virtually-abelian 
if and only if 
$G/H$ admit proper actions of $SL(2,\R)$ (see also Appendix \ref{sec:app}).
We illustrated an idea of the proof 
in \cite[Section 7.5]{Okuda13cls}
by an example $\mathfrak{g} = \mathfrak{su}(4,2)$,
but postponed a full proof to this paper.

\section{Algorithm to classify hyperbolic elements come from nilpotent orbits}

In this section, we recall the algorithm to classify elements in $\mathfrak{a}^{-w_0}$ and $\mathcal{H}^n(\mathfrak{a}_+)$ described in \cite{Okuda13cls}.

\subsection{Notation}

We set up our notation.
Let $\mathfrak{g}_\mathbb{C}$ be a complex semisimple Lie algebra 
and $\mathfrak{g}$ a real form of $\mathfrak{g}_\mathbb{C}$.
We fix a Cartan decomposition $\mathfrak{g} = \mathfrak{k} + \mathfrak{p}$ with a Cartan involution $\theta$ on $\mathfrak{g}$.

Take a $\theta$-stable split Cartan subalgebra $\mathfrak{j}_0$ of $\mathfrak{g}$.
That is, $\mathfrak{j}_0$ is a maximal abelian subspace of $\mathfrak{g}$ stable by $\theta$ such that $\mathfrak{a} := \mathfrak{j}_0 \cap \mathfrak{p}$ is a maximal abelian subspace of $\mathfrak{p}$.
Then $\mathfrak{j}_0$ can be written as $\mathfrak{j}_0 = \mathfrak{t} + \mathfrak{a}$ where $\mathfrak{t}$ is a maximal abelian subspace of the centralizer $Z_\mathfrak{k}(\mathfrak{a})$ of $\mathfrak{a}$ in $\mathfrak{k}$.
Let us denote by $\mathfrak{j}_\mathbb{C} := \mathfrak{j}_0 + \sqrt{-1}\mathfrak{j}_0$ 
and $\mathfrak{j} := \sqrt{-1} \mathfrak{t} + \mathfrak{a}$.
Then $\mathfrak{j}_\mathbb{C}$ is a Cartan subalgebra 
of $\mathfrak{g}_\mathbb{C}$ and $\mathfrak{j}$ is a real form of it with 
\[
\mathfrak{j} = \{ A \in \mathfrak{j}_\mathbb{C} \mid \alpha(A) \in \mathbb{R} \ 
\text{for any } \alpha \in \Delta \},
\]
where $\Delta$ is the root system of $(\mathfrak{g}_\mathbb{C},\mathfrak{j}_\mathbb{C})$. 
We put \[
\Sigma := \{ \alpha|_\mathfrak{a} \mid \alpha \in \Delta \} \setminus \{0\} \subset \mathfrak{a}^*
\] 
to the restricted root system of $(\mathfrak{g},\mathfrak{a})$.
Then we can take a positive system $\Delta^+$ of $\Delta$ such that the subset 
\[
\Sigma^+ := 
\{ \alpha|_\mathfrak{a} \mid \alpha \in \Delta^+ \} \setminus \{0\}.
\]
of $\Sigma$ becomes a positive system.
In fact, if we take an ordering on $\mathfrak{a}$ and extend it to $\mathfrak{j}$, 
then the corresponding positive system $\Delta^+$ satisfies the condition above.
Let us denote by $W^\mathbb{C}$, $W$ the Weyl groups of $\Delta$, $\Sigma$, respectively.
We set the closed positive Weyl chambers
\begin{align*}
\mathfrak{j}_+ &:= \{\, A \in \mathfrak{j} \mid \alpha(A) \geq 0 \quad 
\text{for any } \alpha \in \Delta^+ \,\}, \\
\mathfrak{a}_+ &:= \{\, A \in \mathfrak{a} \mid \xi(A) \geq 0 \quad 
\text{for any } \xi \in \Sigma^+ \,\}.
\end{align*}
Then $\mathfrak{j}_+$ and $\mathfrak{a}_+$ are fundamental domains 
of $\mathfrak{j}$, $\mathfrak{a}$ for the actions of $W^\mathbb{C}$ and $W$, respectively.
By the definition of $\Delta^+$ and $\Sigma^+$, 
we have $\mathfrak{a}_+ = \mathfrak{j}_+ \cap \mathfrak{a}$.

Let $w_0$ denote the longest element of $W$ with respect to
the positive system $\Sigma^+$.
Then the linear transform $x \mapsto -w_0 \cdot x$ on $\mathfrak{a}$ 
leaves the closed Weyl chamber $\mathfrak{a}_+$ invariant.
As in Y.~Benoist \cite{Benoist96}, we define a subspace $\mathfrak{b}$ in $\mathfrak{a}$ by
\begin{align*}
\mathfrak{b} := \mathfrak{a}^{-w_0} = \{\, A \in \mathfrak{a} \mid -w_0 \cdot A = A \,\}.
\end{align*}

\subsection{Satake diagram}
\label{subsection:Satake_and_weighted}

We recall some facts for Satake diagrams of $\mathfrak{g}$ and weighted Dynkin diagrams corresponding to elements in $\mathfrak{a}_+$ in this subsection.

Let us denote by $\Pi$ the fundamental system of $\Delta^+$.
Then 
\[
\overline{\Pi} := \{\, \alpha|_\mathfrak{a} \mid \alpha \in \Pi \,\} \setminus \{0\}
\] is the fundamental system of $\Sigma^+$.
The Satake diagram $S_\mathfrak{g}$ of a semisimple Lie algebra $\mathfrak{g}$ 
consists of the following three data: 
the Dynkin diagram of $\mathfrak{g}_\mathbb{C}$ with nodes $\Pi$; 
black nodes $\Pi_0 := \{ \alpha \in \Pi \mid \alpha|_\mathfrak{a} = 0 \}$ in $S$; 
and arrows joining $\alpha \in \Pi \setminus \Pi_0$ and 
$\beta \in \Pi \setminus \Pi_0$ in $S$ 
whose restrictions to $\mathfrak{a}$ are the same 
(see \cite{Araki62,Satake60} for more details).

For any $A \in \mathfrak{j}$, we can define a map 
\[
\Psi_A : \Pi \rightarrow \mathbb{R},\ \alpha \mapsto \alpha(A).
\]
We call $\Psi_A$ the weighted Dynkin diagram corresponding to $A \in \mathfrak{j}$,
and $\alpha(A)$ the weight on a node $\alpha \in \Pi$ of the weighted Dynkin diagram.
Since $\Pi$ is a basis of $\mathfrak{j}^*$, the correspondence 
\begin{align}
\Psi: \mathfrak{j} \rightarrow \Map(\Pi,\mathbb{R}),\ A \mapsto \Psi_A \label{eq:psi}
\end{align}
is a linear isomorphism between real vector spaces. 
In particular, $\Psi$ is bijective, 
and hence
\[
\Psi|_{\mathfrak{j}_+} : \mathfrak{j}_+ \rightarrow \Map(\Pi,\mathbb{R}_{\geq 0}),\ A \mapsto \Psi_A
\]
is also bijective.
We say that a weighted Dynkin diagram is trivial if all weights are zero.
Namely, the trivial diagram corresponds 
to the zero in $\mathfrak{j}$ by $\Psi$.

Here, we recall the definition of weighted Dynkin diagrams \textit{matching} 
the Satake diagram $S_\mathfrak{g}$ of $\mathfrak{g}$ as follows:

\begin{defn}[{\cite[Definition 7.3]{Okuda13cls}}]\label{defn:match}
Let $\Psi \in \Map(\Pi,\mathbb{R})$ be a weighted Dynkin diagram 
of $\mathfrak{g}_\mathbb{C}$ and 
$S_\mathfrak{g}$ the Satake diagram of $\mathfrak{g}$ with nodes $\Pi$.
We say that $\Psi$ \textit{matches} $S_\mathfrak{g}$ 
if all the weights on black nodes in $\Pi_0$ are zero and 
any pair of nodes joined by an arrow have the same weights.
\end{defn}

Then the following lemma holds:

\begin{lem}[{\cite[Lemma 7.5]{Okuda13cls}}]\label{lem:a_to_Satake}
The linear isomorphism $\Psi:\mathfrak{j} \rightarrow \Map(\Pi,\mathbb{R})$ 
induces a linear isomorphism  
\[
\mathfrak{a} \rightarrow 
\{\, \Psi_A \in \Map(\Pi,\mathbb{R}) \mid \text{$\Psi_A$ matches $S_\mathfrak{g}$} \,\},\quad A \mapsto \Psi_A.
\]
In particular, by this linear isomorphism, 
we have 
\[
\mathfrak{a}_+ \bijarrow 
\{\, \Psi_A \in \Map(\Pi,\mathbb{R}_{\geq 0}) \mid \text{$\Psi_A$ matches $S_\mathfrak{g}$} \,\}.
\]
\end{lem}

\subsection{Weighted Dynkin diagrams corresponding to elements in $\mathfrak{b}$}\label{subsection:WDD_hyp}

In this subsection, we recall an algorithm to classify all weighted Dynkin diagrams corresponding to elements in $\mathfrak{b} = \mathfrak{a}^{-w_0}$ (see \eqref{eq:b}) by using the Satake diagram of $\mathfrak{g}$ and the opposition involution on the Dynkin diagram of $\mathfrak{g}_\C$ studied by Tits \cite{Tits66}.

Let us denote by $w_0^\mathbb{C}$ the longest element of $W^\mathbb{C}$ corresponding to the positive system $\Delta^+$.
Then, by the action of $w_0^\mathbb{C}$, 
every element in $\mathfrak{j}_+$ moves to 
$-\mathfrak{j}_+ := \{ -A \mid A \in \mathfrak{j}_+ \}$.
In particular,
\[
-w_0^\mathbb{C} : \mathfrak{j} \rightarrow \mathfrak{j},\quad A \mapsto -(w_0^\mathbb{C} \cdot A)
\] 
is an involutive automorphism on $\mathfrak{j}$ preserving $\mathfrak{j}_+$.
We put 
\[
\mathfrak{j}^{-w_0^\mathbb{C}} := \{\, A \in \mathfrak{j} \mid -w_0^\mathbb{C} \cdot A = A \,\}.
\]

Recall that the map $\Psi : \mathfrak{j} \rightarrow \Map(\Pi,\mathbb{R})$ 
in Section \ref{subsection:Satake_and_weighted}
is a linear isomorphism.
Thus $-w_0^\mathbb{C}$ induces an involutive endomorphism on $\Map(\Pi,\mathbb{R})$, 
which will be denoted by $\iota$.
In other words, 
the involution $\iota$ on $\Map(\Pi,\R)$ is induced by the opposition involution $\Pi \rightarrow \Pi, \alpha \mapsto -(w_0^\C)^* \cdot \alpha$ (see Tits \cite[Section 1.5.1]{Tits66} for more details).

Then we have 
\[
\Psi(\mathfrak{j}^{-w_0^\mathbb{C}}) = \Map(\Pi,\mathbb{R})^\iota,
\]
where $\Map(\Pi,\mathbb{R})^\iota$ denote 
the set of all weighted Dynkin diagrams held invariant by $\iota$.
For each complex simple Lie algebra $\mathfrak{g}_\mathbb{C}$,
we determine $\iota$ as follows:

\begin{prop}[{\cite[Theorem 6.3]{Okuda13cls}}]\label{thm:Bhyp_to_iota-inv}
Suppose that $\mathfrak{g}_\mathbb{C}$ is simple.
The involution $\iota$ is not identity 
if and only if $\mathfrak{g}_\mathbb{C}$ is 
of type $A_n$, $D_{2k+1}$ or $E_6$ $($$n \geq 2,\ k \geq 2$$)$.
In other words, this is the complete list of simple $\mathfrak{g}_\mathbb{C}$
with $\mathfrak{j}^{-w_0^\mathbb{C}} \neq \mathfrak{j}$. 
In such cases, the forms of $\iota$ are the following$:$
\begin{description}
\item[For type $A_n$ $($$n \geq 2$, $\mathfrak{g}_\mathbb{C} \simeq \mathfrak{sl}(n+1,\mathbb{C})$$)$]
\[
\iota :
\begin{xy}
	*++!D{a_1} *\cir<2pt>{}        ="A",
	(10,0) *++!D{a_2} *\cir<2pt>{} ="B",
	(30,0) *++!D{a_{n-1}} *\cir<2pt>{} ="D",
	(40,0) *++!D{a_n} *\cir<2pt>{} ="E",
	\ar@{-} "A";"B"
	\ar@{-} "B";(15,0)
	\ar@{.} (15,0);(25,0)
	\ar@{-} (25,0);"D"
	\ar@{-} "D";"E"
\end{xy}
\mapsto 
\begin{xy}
	*++!D{a_n} *\cir<2pt>{}        ="A",
	(10,0) *++!D{a_{n-1}} *\cir<2pt>{} ="B",
	(30,0) *++!D{a_2} *\cir<2pt>{} ="D",
	(40,0) *++!D{a_1} *\cir<2pt>{} ="E",
	\ar@{-} "A";"B"
	\ar@{-} "B";(15,0)
	\ar@{.} (15,0);(25,0)
	\ar@{-} (25,0);"D"
	\ar@{-} "D";"E"
\end{xy}
\]
\item [For type $D_{2k+1}$ $($$k \geq 2$, $\mathfrak{g}_\mathbb{C} \simeq \mathfrak{so}(4k+2,\mathbb{C})$$)$] 
\[
\iota : 
\begin{xy}
	*++!D{a_1} *\cir<2pt>{}        ="A",
	(10,0) *++!D{a_2} *\cir<2pt>{} ="B",
	(25,0) *++!D{a_{2k-1}} *\cir<2pt>{} ="C",
	(35,-5) *++!D{a_{2k}} *++!U{} *\cir<2pt>{} ="D",
	(35,5) *++!D{a_{2k+1}} *\cir<2pt>{} ="E",	
	\ar@{-} "A";"B"
	\ar@{-} "B"; (15,0)
	\ar@{.} (15,0) ; (20,0)^*!U{\cdots}
	\ar@{-} (20,0) ; "C"
	\ar@{-} "C";"D"
	\ar@{-} "C";"E"
\end{xy}
\mapsto
\begin{xy}
	*++!D{a_1} *\cir<2pt>{}        ="A",
	(10,0) *++!D{a_2} *\cir<2pt>{} ="B",
	(25,0) *++!D{a_{2k-1}} *\cir<2pt>{} ="C",
	(35,-5) *++!D{a_{2k+1}} *++!U{} *\cir<2pt>{} ="D",
	(35,5) *++!D{a_{2k}} *\cir<2pt>{} ="E",	
	\ar@{-} "A";"B"
	\ar@{-} "B"; (15,0)
	\ar@{.} (15,0) ; (20,0)^*!U{\cdots}
	\ar@{-} (20,0) ; "C"
	\ar@{-} "C";"D"
	\ar@{-} "C";"E"
\end{xy}
\]
\item[For type $E_6$ $($$\mathfrak{g}_\mathbb{C} \simeq \mathfrak{e}_{6,\mathbb{C}}$$)$]
\[
\iota : 
\begin{xy}
	*++!D{a_1} *\cir<2pt>{}        ="A",
	(10,0) *++!D{a_2} *\cir<2pt>{} ="B",
	(20,0) *++!D{a_3} *\cir<2pt>{} ="C",
	(30,0) *++!D{a_4} *\cir<2pt>{} ="D",
	(40,0) *++!D{a_5} *\cir<2pt>{} ="E",
	(20,-10) *++!R{a_6} *\cir<2pt>{} ="F",
	\ar@{-} "A";"B"
	\ar@{-} "B";"C"
	\ar@{-} "C";"D"
	\ar@{-} "D";"E"
	\ar@{-} "C";"F"
\end{xy}
\mapsto
\begin{xy}
	*++!D{a_5} *\cir<2pt>{}        ="A",
	(10,0) *++!D{a_4} *\cir<2pt>{} ="B",
	(20,0) *++!D{a_3} *\cir<2pt>{} ="C",
	(30,0) *++!D{a_2} *\cir<2pt>{} ="D",
	(40,0) *++!D{a_1} *\cir<2pt>{} ="E",
	(20,-10) *++!R{a_6} *\cir<2pt>{} ="F",
	\ar@{-} "A";"B"
	\ar@{-} "B";"C"
	\ar@{-} "C";"D"
	\ar@{-} "D";"E"
	\ar@{-} "C";"F"
\end{xy}
\]
\end{description}
\end{prop}

It should be noted that 
for the case where $\mathfrak{g}_\mathbb{C}$ is of type $D_{2k}$ ($k \geq 2$), 
the involution $-w_0^\mathbb{C}$ on $\mathfrak{j}$ is trivial 
although the Dynkin diagram of type $D_{2k}$ admits some non-trivial involutive automorphisms.

To classify elements in $\mathfrak{b}$, we use the following lemma:
\begin{lem}[{\cite[Lemma 7.6]{Okuda13cls}}]\label{lem:bjw}
$
\mathfrak{b} = \mathfrak{j}^{-w_0^\mathbb{C}} \cap \mathfrak{a}.
$
\end{lem}

By combining Lemma \ref{lem:a_to_Satake} with Lemma \ref{lem:bjw},
we obtain that 
\begin{align}
\Psi(\mathfrak{b}) 
= \{\, \Psi_A \in \Map(\Pi,\mathbb{R})^\iota \mid \text{$\Psi_A$ matches $S_\mathfrak{g}$} \,\}. \label{eq:b}
\end{align}
where $\Map(\Pi,\mathbb{R})^\iota$ denote 
the set of all weighted Dynkin diagrams held invariant by $\iota$,
and $S_g$ is the Satake diagram of $\mathfrak{g}$ 
(see Section \ref{subsection:Satake_and_weighted}).

\subsection{The Dynkin--Kostant classification}\label{subsection:DKcls}

In this subsection, we recall the Dynkin--Kostant classification of complex nilpotent orbits in a complex semisimple Lie algebra $\mathfrak{g}_\C$.

Let us denote by $\Hom(\mathfrak{sl}(2,\mathbb{C}),\mathfrak{g}_\mathbb{C})$ 
the set of all complex Lie algebra homomorphisms 
from $\mathfrak{sl}(2,\mathbb{C})$ to $\mathfrak{g}_\mathbb{C}$.
For each $\rho \in \Hom(\mathfrak{sl}(2,\mathbb{C}),\mathfrak{g}_\mathbb{C})$,
there uniquely exists an element $A_\rho$ of $\mathfrak{j}_+$ 
which is conjugate to the element 
\[
\rho \begin{pmatrix} 1 & 0 \\ 0 & -1\end{pmatrix} \in \mathfrak{g}_\mathbb{C}
\]
under the adjoint action on $\mathfrak{g}_\mathbb{C}$.
Then the correspondence $[\rho] \mapsto A_\rho$ gives a map 
from $\Hom(\mathfrak{sl}(2,\mathbb{C}),\mathfrak{g}_\mathbb{C})/{\Int \mathfrak{g}_\C}$
to $\mathfrak{j}_+$.
Malcev \cite{Malcev50} proved that the map 
\[
\Hom(\mathfrak{sl}(2,\mathbb{C}),\mathfrak{g}_\mathbb{C})/{\Int \mathfrak{g}_\C} \hookrightarrow \mathfrak{j}_+,~ [\rho] \mapsto A_\rho
\]
is injective.
We put $\mathcal{H}^n(\mathfrak{j}_+)$ the image of the injective map above.
Obviously, we have 
\[
\mathcal{H}^n(\mathfrak{j}_+) = \mathfrak{j}_+ \cap \left\{ \rho \begin{pmatrix}
1 & 0 \\ 0 & -1 
\end{pmatrix} \mid \rho \in \Hom(\mathfrak{sl}(2,\C), \mathfrak{g}_\C) 
\right\}.
\]

Recall that $\Psi$ in Section \ref{subsection:Satake_and_weighted}
induces a bijection between $\mathfrak{j}_+$ and $\Map(\Pi,\mathbb{R}_{\geq 0})$.
Thus, a classification of $\Psi(\mathcal{H}^n(\mathfrak{j}_+))$ gives 
that of $\Hom(\mathfrak{sl}(2,\mathbb{C}),\mathfrak{g}_\mathbb{C})/{\Int \mathfrak{g}_\C}$.
Dynkin \cite{Dynkin52eng} proved that 
any weight of 
an weighted Dynkin diagram in $\Psi(\mathcal{H}^n(\mathfrak{j}_+))$ 
is given by $0$, $1$ or $2$.
Hence, $\Psi(\mathcal{H}^n(\mathfrak{j}_+))$ 
(and therefore the set $\Hom(\mathfrak{sl}(2,\mathbb{C}),\mathfrak{g}_\mathbb{C})/{\Int \mathfrak{g}_\C}$ is) finite.
Dynkin \cite{Dynkin52eng} also gave 
a complete list of the weighted Dynkin diagrams in $\Psi(\mathcal{H}^n(\mathfrak{j}_+))$
for each simple $\mathfrak{g}_\mathbb{C}$.

By combining the Jacobson--Morozov theorem with the results of Kostant \cite{Kostant59},
we also obtain a bijection 
\[
\Hom(\mathfrak{sl}(2,\mathbb{C}),\mathfrak{g}_\mathbb{C})/{\Int \mathfrak{g}_\C} \simrightarrow \mathcal{N}(\mathfrak{g}_\C)/{\Int \mathfrak{g}_\C},~ [\rho] \mapsto \Int \mathfrak{g}_\C \cdot \rho 
\begin{pmatrix}
0 & 1 \\ 0 & 0 
\end{pmatrix}.
\]
Here $\mathcal{N}(\mathfrak{g}_\C)/{\Int \mathfrak{g}_\C}$ denote the set of all complex nilpotent adjoint orbits in $\mathfrak{g}_\C$.
Thus, the classification of $\Psi(\mathcal{H}^n(\mathfrak{j}_+))$, 
done by Dynkin \cite{Dynkin52eng},
gives a classification of complex nilpotent adjoint orbits in $\mathfrak{g}_\mathbb{C}$ $($see Bala--Cater \cite{Bala-Carter76} or 
Collingwood--McGovern \cite[\text{Section } 3]{Collingwood-McGovern93}
for more details$)$.
This is known as the Dynkin--Kostant classification of complex nilpotent adjoint orbits.

\subsection{Algorithm to classify hyperbolic elements come from nilpotent orbits}\label{subsection:class_nilp}

In this subsection, we give an algorithm to determine the finite subset 
\[
\mathcal{H}^n(\mathfrak{a}_+) := \mathfrak{a}_+ \cap \left\{ \rho \begin{pmatrix}
1 & 0 \\ 0 & -1 
\end{pmatrix} \mid \rho \in \Hom(\mathfrak{sl}(2,\R), \mathfrak{g}) \right\}
\] of $\mathfrak{a}_+$
where $\Hom(\mathfrak{sl}(2,\R), \mathfrak{g})$ is the set of all Lie algebra homomorphisms from $\mathfrak{sl}(2,\R)$ to $\mathfrak{g}$ as in Section \ref{section:purpose}.
We note that similar to the situation for $\mathfrak{g}_\C$, 
for any $\rho \in \Hom(\mathfrak{sl}(2,\R),\mathfrak{g})$,
there uniquely exists $A_\rho \in \mathcal{H}^n(\mathfrak{a}_+)$ such that 
$\rho \begin{pmatrix}
1 & 0 \\ 0 & -1
\end{pmatrix}$ is conjugate to $A_\rho$ under the adjoint action on $\mathfrak{g}$. 

Recall that $\mathfrak{g} = \mathfrak{k} + \mathfrak{p}$ is a real form of $\mathfrak{g}_\C$, $\mathfrak{j} = \sqrt{-1} \mathfrak{t} + \mathfrak{a}$ and $\mathfrak{a}_+ = \mathfrak{a} \cap \mathfrak{j}_+$ in our setting.
Then by \cite[Proposition 4.5 (iii)]{Okuda13cls}, we have $\mathcal{H}^n(\mathfrak{a}_+) = \mathfrak{a}_+ \cap \mathcal{H}^n(\mathfrak{j}_+)$.
Therefore, by Lemma \ref{lem:a_to_Satake},
the following holds
\begin{align}
\Psi(\mathcal{H}^n(\mathfrak{a}_+)) = 
\{ \Psi_A \in \Psi(\mathcal{H}^n(\mathfrak{j}_+)) \mid \text{$\Psi_A$ matches $S_\mathfrak{g}$}\}, \label{eq:nilp_a_+}
\end{align}
where $S_\mathfrak{g}$ is the Satake diagram of $\mathfrak{g}$ 
(see Section \ref{subsection:Satake_and_weighted} for the notation).
Hence, for each $\mathfrak{g}$,
by using the classification of $\Psi(\mathcal{H}^n(\mathfrak{j}_+))$ (see Section \ref{subsection:DKcls})
and the Satake diagram $S_\mathfrak{g}$ of $\mathfrak{g}$,
we obtain a classification of $\Psi(\mathcal{H}^n(\mathfrak{a}_+))$.

\begin{rem}
The finite set $\mathcal{H}^n(\mathfrak{a}_+)$ parametrizes the set $\mathcal{N}_\mathfrak{g}(\mathfrak{g}_\C)/{\Int \mathfrak{g}_\C}$ of complex nilpotent adjoint orbits in $\mathfrak{g}_\C$ that meet the real form $\mathfrak{g}$,
i.e.~there is a natural bijection \[
\varphi : \mathcal{H}^n(\mathfrak{a}_+) \simrightarrow \mathcal{N}_\mathfrak{g}(\mathfrak{g}_\C)/{\Int \mathfrak{g}_\C}, \quad 
\rho \begin{pmatrix}
1 & 0 \\ 
0 & -1
\end{pmatrix} \mapsto 
(\Int \mathfrak{g}_\C) \cdot \rho \begin{pmatrix}
0 & 1 \\ 
0 & 0
\end{pmatrix}
\]
$($see \cite[Section 7.4]{Okuda13cls} for more details$)$.
We note that $\mathcal{H}^n(\mathfrak{a}_+)$ is not bijective to the set $\mathcal{N}(\mathfrak{g})/{\Int \mathfrak{g}}$ of real nilpotent adjoint orbits in $\mathfrak{g}$.
In fact, the surjective map $\phi : \mathcal{N}(\mathfrak{g})/{\Int \mathfrak{g}} \twoheadrightarrow \mathcal{H}^n(\mathfrak{a}_+)$ defined in Section \ref{section:purpose} is not always injective. 
In summary, we have the following commutative diagram:
\[\xymatrix{
\mathcal{N}_\mathfrak{g}(\mathfrak{g}_\C)/{\Int \mathfrak{g}_\C} \ar@{<-}^{\quad \sim}_{\quad \varphi}[r] & \mathcal{H}^n(\mathfrak{a}_+) \\ 
\mathcal{N}(\mathfrak{g})/{\Int \mathfrak{g}} \ar@{->>}[u]^{\text{complexification}} \ar@{->>}[ur]_{\phi} &  
}\]
where the complexification of a real nilpotent adjoint orbit $\mathcal{O}$ in $\mathfrak{g}$ is defined as $(\Int \mathfrak{g}_\C) \cdot \mathcal{O} \subset \mathfrak{g}_\C$.
The survey of classifications of real nilpotent (adjoint) orbits in semisimple Lie algebras can be found in \cite[Chapter 9]{Collingwood-McGovern93}.
\end{rem}

\section{Proof of Theorem \ref{thm:intro}}\label{section:proof}

First, we show the easy part of Theorem \ref{thm:intro} as follows:

\begin{lem}\label{lem:nilp_in_b}
$\mathcal{H}^n(\mathfrak{a}_+) \subset \mathfrak{b} ~(:= \mathfrak{a}^{-w_0})$.
\end{lem}

\begin{proof}[Proof of Lemma $\ref{lem:nilp_in_b}$]
Let us fix any $\rho \in \Hom(\mathfrak{sl}(2,\R),\mathfrak{g})$ with 
$A_\rho := \rho \begin{pmatrix}
1 & 0 \\ 0 & -1
\end{pmatrix} \in \mathfrak{a}_+$.
It is enough to show that $-A_\rho = w_0 A_\rho$.
We denote by $\tilde{\rho} : PSL(2,\R) \rightarrow \Int \mathfrak{g}$ the Lie group homomorphism corresponding to $\rho : \mathfrak{sl}(2,\R) \rightarrow \mathfrak{g}$.

One can observe that, in $\mathfrak{sl}(2,\R)$, 
the two elements 
$\begin{pmatrix}
1 & 0 \\ 0 & -1
\end{pmatrix}$
and 
$\begin{pmatrix}
-1 & 0 \\ 0 & 1
\end{pmatrix}$
are conjugate under the adjoint action of $PSL(2,\R)$.
Therefore two elements $A_\rho, -A_\rho \in \mathfrak{g}$ are conjugate under the adjoint action of $\tilde{\rho}(PSL(2,\R)) \subset \Int \mathfrak{g}$.
On the other hand, 
for any $A \in \mathfrak{a}_+ \subset \mathfrak{g}$,
we have that \[
(\Int \mathfrak{g} \cdot A) \cap \mathfrak{a} = W \cdot A
\]
(see \cite[Lemma 7.2]{Okuda13cls} for more details).
Hence for $A_\rho \in \mathfrak{a}_+$, the element $-A_\rho$ is in $W \cdot A_\rho$.
Since $-A_\rho \in -\mathfrak{a}_+$, we have $- A_\rho = w_0 A_\rho$.
\end{proof}

In the rest of this section, 
we give a proof of non-trivial part 
$\mathfrak{b} \subset \mathbb{R}\text{-}\mathrm{span} (\mathcal{H}^n(\mathfrak{a}_+))$
of Theorem \ref{thm:intro}.

Recall that $\Psi : \mathfrak{j} \rightarrow \Map(\Pi,\mathbb{R})$ 
in Section \ref{subsection:Satake_and_weighted}
is a linear isomorphism.
Therefore, to complete the proof of Theorem \ref{thm:intro}, we only need to show that 
\begin{equation}
\Psi(\mathfrak{b}) \subset \mathbb{R}\text{-}\mathrm{span} \Psi(\mathcal{H}^n(\mathfrak{a}_+)) 
\label{eq:goal}
\end{equation}

We note that our claim for semisimple $\mathfrak{g}$ 
is reduced to that on each simple factor of $\mathfrak{g}$.
Furthermore, in the case where $\mathfrak{g}$ is a complex simple Lie algebra, 
our claim is reduced to the case where $\mathfrak{g}'$ is a split real form of $\mathfrak{g}$.
Thus, it is enough to show \eqref{eq:goal} 
for the case where 
$\mathfrak{g}_\mathbb{C}$ is simple and 
$\mathfrak{g}$ is non-compact real form of $\mathfrak{g}_\mathbb{C}$.
To do this, we shall find weighted Dynkin diagrams 
$\Psi_1,\dots,\Psi_n$ in $\Psi(\mathcal{H}^n(\mathfrak{a}_+))$
such that $\{ \Psi_1,\dots, \Psi_n \}$ 
becomes a basis of $\Psi(\mathfrak{b})$ for each non-compact real simple Lie algebra $\mathfrak{g}$.

In the rest of this section,
for each complex simple $\mathfrak{g}_\mathbb{C}$, 
we give an explicit form of $\Map(\Pi,\mathbb{R})^\iota$
(see Section \ref{subsection:WDD_hyp} for the definition of $\iota$)
and some examples in $\Psi(\mathcal{H}^n(\mathfrak{j}_+))$ from the list given by Dynkin \cite{Dynkin52eng} 
(we will refer \cite{Collingwood-McGovern93} for classifications of weighted Dynkin diagrams in $\Psi(\mathcal{H}^n(\mathfrak{j}_+))$).
Then 
for each non-compact real form $\mathfrak{g}$ of $\mathfrak{g}_\mathbb{C}$,
we give the Satake diagram $S_\mathfrak{g}$ of $\mathfrak{g}$,
which can be found in \cite{Araki62}, and the explicit form of 
$\Psi(\mathfrak{b})$ by using \eqref{eq:b} in Section \ref{subsection:WDD_hyp}.
Finally, for each $\mathfrak{g}$, 
we give an example of a basis $\{ \Psi_1,\dots,\Psi_n \}$ of $\Psi(\mathfrak{b})$ 
with 
\[
\Psi_i \in \Psi(\mathcal{H}^n(\mathfrak{a}_+)) = 
\{ \Psi_A \in \Psi(\mathcal{H}^n(\mathfrak{j}_+)) \mid \text{$\Psi_A$ matches $S_\mathfrak{g}$}\} 
\]
(see Section \ref{subsection:class_nilp} for the details).
Then the proof of Theorem \ref{thm:intro} will be completed.

\begin{rem}\label{rem:even}
As in the following subsections,
our basis $\{ \Psi_1,\dots,\Psi_n \}$ of $\Psi(\mathfrak{b})$ consists of weighted Dynkin diagrams of even nilpotent orbits,
where ``even'' means that any weight of $\Psi_i$ is $0$ or $2$.
\end{rem}

\subsection{Type $A_l$}

Let us consider the case where 
$\mathfrak{g}_\mathbb{C}$ is of type $A_l$ for $l \geq 1$,
that is,
$\mathfrak{g}_\mathbb{C} \simeq \mathfrak{sl}(l+1,\mathbb{C})$.
Then we have
\[
\Map(\Pi,\mathbb{R})^\iota 
= \left\{ \begin{xy}
	*++!D{a_1} *\cir<2pt>{}        ="A",
	(10,0) *++!D{a_2} *\cir<2pt>{} ="B",
	(30,0) *++!D{a_{l-1}} *\cir<2pt>{} ="D",
	(40,0) *++!D{a_l} *\cir<2pt>{} ="E",
	\ar@{-} "A";"B"
	\ar@{-} "B";(15,0)
	\ar@{.} (15,0);(25,0)^*U{\cdots}
	\ar@{-} (25,0);"D"
	\ar@{-} "D";"E"
\end{xy} ~\middle|~ 
a_i = a_{l+1-i} \text{ for } i =1 ,\dots,l
\right\}
\]

By \cite[Section 3.6]{Collingwood-McGovern93}, 
we can find some examples of weighted Dynkin diagrams in 
$\Psi(\mathcal{H}^n(\mathfrak{j}_+))$ as follows:

\begin{center}
\begin{longtable}{ll} \toprule
	Symbol & Weighted Dynkin diagram in $\Psi(\mathcal{H}^n(\mathfrak{j}_+))$ \\ \midrule
	$[l+1]$ & \begin{xy}
	*++!D{2} *\cir<2pt>{}        ="A",
	(6,0) *++!D{2} *\cir<2pt>{} ="B",
	(15,0) *++!D{2} *\cir<2pt>{} ="D",
	(21,0) *++!D{2} *\cir<2pt>{} ="E",
	\ar@{-} "A";"B"
	\ar@{-} "B";(9,0)
	\ar@{.} (9,0);(12,0)^*U{\cdots}
	\ar@{-} (12,0);"D"
	\ar@{-} "D";"E"
\end{xy}  \\ 	
	
	$[2s+1,1^{l-2s}]$ & \begin{xy}
	*++!D{2} *\cir<2pt>{}        ="A",
	(3,0) = "A_1",
	(8,0) = "A_2",
	(11,0) *++!D{2} *\cir<2pt>{} ="B",
	(17,0) *++!D{2} *++!U{\alpha_s} *\cir<2pt>{} ="C",
	(23,0) *++!D{0} *\cir<2pt>{} ="D",
	(26,0) = "D_1",
	(31,0) = "D_2",
	(34,0) *++!D{0} *\cir<2pt>{} ="E",
	(40,0) *++!D{2} *++!U{\alpha_{l+1-s}} *\cir<2pt>{} ="F",
	(46,0) *++!D{2} *\cir<2pt>{} ="G",
	(49,0) = "G_1",
	(54,0) = "G_2",
	(57,0) *++!D{2} *\cir<2pt>{} ="H",
	\ar@{-} "A";"A_1"
	\ar@{.} "A_1";"A_2"^*!U{\cdots}
	\ar@{-} "A_2";"B"
	\ar@{-} "B";"C"
	\ar@{-} "C";"D"
	\ar@{-} "D";"D_1"
	\ar@{.} "D_1";"D_2"^*!U{\cdots}
	\ar@{-} "D_2";"E"
	\ar@{-} "E";"F"
	\ar@{-} "F";"G"
	\ar@{-} "G";"G_1"
	\ar@{.} "G_1";"G_2"^*!U{\cdots}
	\ar@{-} "G_2";"H"
\end{xy} \\ 	
	$[(2s+1)^2,1^{l-4s-1}]$ & \begin{xy}
	(0,0) *++!D{0} *\cir<2pt>{}        ="AAAA",
	(3,0) *++!D{2} *\cir<2pt>{}        ="AAA",
	(6,0)*++!D{0} *\cir<2pt>{}        ="AA",
	(9,0)*++!D{2} *\cir<2pt>{}        ="A",
	(12,0) = "A_1",
	(17,0) = "A_2",
	(20,0) *++!D{0} *\cir<2pt>{} ="B",
	(23,0) *++!D{2} *++!U{\alpha_{2s}} *\cir<2pt>{} ="C",
	(26,0) *++!D{0} *\cir<2pt>{} ="DD",
	(29,0) *++!D{0} *\cir<2pt>{} ="D",
	(32,0) = "D_1",
	(37,0) = "D_2",
	(40,0) *++!D{0} *\cir<2pt>{} ="E",
	(43,0) *++!D{0} *\cir<2pt>{} ="EE",
	(46,0) *++!D{2} *++!U{\alpha_{l+1-2s}} *\cir<2pt>{} ="F",
	(49,0) *++!D{0} *\cir<2pt>{} ="G",
	(52,0) = "G_1",
	(57,0) = "G_2",
	(60,0) *++!D{2} *\cir<2pt>{} ="H",
	(63,0) *++!D{0} *\cir<2pt>{} ="HH",
	(66,0) *++!D{2} *\cir<2pt>{} ="HHH",
	(69,0) *++!D{0} *\cir<2pt>{} ="HHHH",
	\ar@{-} "AAAA";"AAA"
	\ar@{-} "AAA";"AA"
	\ar@{-} "AA";"A"
	\ar@{-} "A";"A_1"
	\ar@{.} "A_1";"A_2"^*!U{\cdots}
	\ar@{-} "A_2";"B"
	\ar@{-} "B";"C"
	\ar@{-} "C";"DD"
	\ar@{-} "DD";"D"
	\ar@{-} "D";"D_1"
	\ar@{.} "D_1";"D_2"^*!U{\cdots}
	\ar@{-} "D_2";"E"
	\ar@{-} "E";"EE"
	\ar@{-} "EE";"F"
	\ar@{-} "F";"G"
	\ar@{-} "G";"G_1"
	\ar@{.} "G_1";"G_2"^*!U{\cdots}
	\ar@{-} "G_2";"H"
	\ar@{-} "H";"HH"
	\ar@{-} "HH";"HHH"
	\ar@{-} "HHH";"HHHH"
\end{xy} \\ 
\bottomrule
\end{longtable}
\end{center}

Let $\mathfrak{g}$ be a non-compact real form of $\mathfrak{g}_\mathbb{C}$.
Then the Satake diagram $S_\mathfrak{g}$ and $\Psi(\mathfrak{b})$ are given as follows:

\begin{center}
	\begin{longtable}{lll} \toprule
	$\mathfrak{g}$ & $S_\mathfrak{g}$ & $\Psi(\mathfrak{b})$ \\ \midrule
	
$\mathfrak{sl}(l+1,\mathbb{R})$ & \begin{xy}
	*\cir<2pt>{}        ="A",
	(6,0) *\cir<2pt>{} ="B",
	(15,0) *\cir<2pt>{} ="C",
	(21,0) *\cir<2pt>{} ="D",
	\ar@{-} "A";"B"
	\ar@{-} "B"; (9,0)
	\ar@{.} (9,0) ; (12,0)
	\ar@{-} (12,0) ; "C"
	\ar@{-} "C";"D"
\end{xy} 
& $\left\{ \begin{xy}
	*++!D{b_1} *\cir<2pt>{}        ="A",
	(6,0) *++!D{b_2} *\cir<2pt>{} ="B",
	(15,0) *++!D{b_2} *\cir<2pt>{} ="D",
	(21,0) *++!D{b_1} *\cir<2pt>{} ="E",
	\ar@{-} "A";"B"
	\ar@{-} "B";(9,0)
	\ar@{.} (9,0);(12,0)^*U{\cdots}
	\ar@{-} (12,0);"D"
	\ar@{-} "D";"E"
\end{xy} 
\right\}$
\\ \hline
	
\raisebox{-0.3cm}{\shortstack{$\mathfrak{su}^*(2k)$ \\ $(2k=l+1)$}} & \begin{xy}
	 *{\bullet}        ="A",
	(3,0) *\cir<2pt>{} ="B",
	(6,0) *{\bullet} ="C",
	(9,0) *\cir<2pt>{} ="C'",
	(15,0) *\cir<2pt>{} ="D'",
	(18,0) *{\bullet} ="D",
	(21,0) *\cir<2pt>{} ="E",
	(24,0) *+!D{\alpha_{2k-1}} *{\bullet} ="F",
	\ar@{-} "A";"B"
	\ar@{-} "B";"C"
	\ar@{-} "C";"C'"
	\ar@{-} "C'"; (11,0)
	\ar@{.} (11,0) ; (13,0)
	\ar@{-} (13,0) ; "D'"
	\ar@{-} "D'";"D"
	\ar@{-} "D";"E"
	\ar@{-} "E";"F"
\end{xy} & $\left\{ 
\begin{xy}
	*++!D{0} *\cir<2pt>{}   ="A",
	(4,0) *++!D{b_1} *\cir<2pt>{} ="B",
	(8,0) *++!D{0} *\cir<2pt>{} ="C",
	(12,0) *++!D{b_2} *\cir<2pt>{} ="C'",
	(21,0) *++!D{b_{2}} *\cir<2pt>{} ="D'",
	(25,0) *++!D{0} *\cir<2pt>{} ="D",
	(29,0) *++!D{b_{1}} *\cir<2pt>{} ="E",
	(33,0) *++!D{0} *\cir<2pt>{} ="F",
	\ar@{-} "A";"B"
	\ar@{-} "B";"C"
	\ar@{-} "C";"C'"
	\ar@{-} "C'"; (15,0)
	\ar@{.} (15,0) ; (18,0)^*U{\cdots}
	\ar@{-} (18,0) ; "D'"
	\ar@{-} "D'";"D"
	\ar@{-} "D";"E"
	\ar@{-} "E";"F"
\end{xy}
\right\}$
\\ \hline

\raisebox{-0.3cm}{\shortstack{$\mathfrak{su}(p,q)$ \\ $(p+q = l+1)$ \\ $(p>q+1)$}} & 
\raisebox{1cm}{
\begin{xy}
	*+!D{\alpha_1} *\cir<2pt>{} ="A",
	(6,0) *\cir<2pt>{} ="B",
	(15,0) *+!D{\alpha_q} *\cir<2pt>{} ="C",
	(21,0)  *{\bullet} ="D",
	(27,-4) *{\bullet} ="E",
	(27,-13) *{\bullet} ="F",
	(21,-17) *{\bullet} ="G",
	(15,-17) *+!U{\alpha_{l+1-q}} *\cir<2pt>{} ="H",
	(6,-17) *\cir<2pt>{} ="I",
	(0,-17) *+!U{\alpha_{l}} *\cir<2pt>{} ="J",
	\ar@{-} "A";"B"
	\ar@{-} "B"; (9,0)
	\ar@{.} (9,0) ; (12,0)
	\ar@{-} (12,0) ;"C"
	\ar@{-} "C";"D"
	\ar@{-} "D";"E"
    \ar@{-} "E";(27,-7)
    \ar@{.} (27,-7);(27,-10)
    \ar@{-} (27,-10);"F"
	\ar@{-} "G";"F"
	\ar@{-} "H";"G"
	\ar@{-} (12,-17) ;"H"
	\ar@{.} (12,-17) ; (9,-17)
	\ar@{-} "I"; (9,-17)
	\ar@{-} "J";"I"
	\ar@{<->} (0,-2);(0,-15)
	\ar@{<->} (6,-2);(6,-15)
	\ar@{<->} (15,-2);(15,-15)
\end{xy} 
}
& 
$\left\{ 
\raisebox{1cm}{
\begin{xy}
	*+!D{b_1} *\cir<2pt>{} ="A",
	(6,0) *+!D{b_2} *\cir<2pt>{} ="B",
	(15,0) *+!D{b_q} *\cir<2pt>{} ="C",
	(21,0)  *++!D{0} *\cir<2pt>{} ="D",
	(27,-4) *++!L{0} *\cir<2pt>{} ="E",
	(27,-13) *++!L{0} *\cir<2pt>{} ="F",
	(21,-17) *++!D{0} *\cir<2pt>{} ="G",
	(15,-17) *+!U{b_{q}} *\cir<2pt>{} ="H",
	(6,-17) *+!U{b_2} *\cir<2pt>{} ="I",
	(0,-17) *+!U{b_{1}} *\cir<2pt>{} ="J",
	\ar@{-} "A";"B"
	\ar@{-} "B"; (9,0)
	\ar@{.} (9,0) ; (12,0)^*U{\cdots}
	\ar@{-} (12,0) ;"C"
	\ar@{-} "C";"D"
	\ar@{-} "D";"E"
    \ar@{-} "E";(27,-7)
    \ar@{.} (27,-7);(27,-10)
    \ar@{-} (27,-10);"F"
	\ar@{-} "G";"F"
	\ar@{-} "H";"G"
	\ar@{-} (12,-17) ;"H"
	\ar@{.} (12,-17) ; (9,-17)
	\ar@{-} "I"; (9,-17)
	\ar@{-} "J";"I"
\end{xy}
}
\right\}$
\\ \hline	
\raisebox{-0.3cm}{\shortstack{$\mathfrak{su}(k+1,k)$ \\ $(2k = l)$}} & 
\raisebox{0.5cm}{
\begin{xy}
	*+!D{\alpha_1} *\cir<2pt>{} ="A",
	(6,0) *\cir<2pt>{} ="B",
	(15,0)  *+!LD{\alpha_{k}} *\cir<2pt>{} ="C",
	(15,-10) *+!LU{\alpha_{k+1}} *\cir<2pt>{} ="E",
	(6,-10) *\cir<2pt>{} ="F",
	(0,-10) *+!U{\alpha_{2k-1}} *\cir<2pt>{} ="G",
	\ar@{-} "A";"B"
	\ar@{-} "B";(9,0) 
	\ar@{.} (9,0);(12,0)
	\ar@{-} (12,0);"C"
	\ar@/^2mm/@{-} "C";"E"
	\ar@{-} "E";(12,-10)
	\ar@{.} (12,-10);(9,-10)
	\ar@{-} (9,-10);"F"
	\ar@{-} "F";"G"
	\ar@{<->} (0,-2);(0,-8)
	\ar@{<->} (6,-2);(6,-8)
	\ar@{<->} (15,-2);(15,-8)
\end{xy}
} 
& $\left\{ 
\raisebox{0.5cm}{
\begin{xy}
	*+!D{b_1} *\cir<2pt>{} ="A",
	(6,0) *+!D{b_2} *\cir<2pt>{} ="B",
	(15,0)  *+!LD{b_{k}} *\cir<2pt>{} ="C",
	(15,-10) *+!LU{b_{k}} *\cir<2pt>{} ="E",
	(6,-10) *+!U{b_{2}} *\cir<2pt>{} ="F",
	(0,-10) *+!U{b_1} *\cir<2pt>{} ="G",
	\ar@{-} "A";"B"
	\ar@{-} "B";(9,0) 
	\ar@{.} (9,0);(12,0)^*U{\cdots}
	\ar@{-} (12,0);"C"
	\ar@/^2mm/@{-} "C";"E"
	\ar@{-} "E";(12,-10)
	\ar@{.} (12,-10);(9,-10)
	\ar@{-} (9,-10);"F"
	\ar@{-} "F";"G"
\end{xy}}
\right\}$
\\ \hline
\raisebox{-0.3cm}{\shortstack{$\mathfrak{su}(k,k)$ \\ $(2k = l+1)$}} & 
\raisebox{0.5cm}{
\begin{xy}
	*+!D{\alpha_1} *\cir<2pt>{} ="A",
	(6,0) *\cir<2pt>{} ="B",
	(15,0)  *\cir<2pt>{} ="C",
	(21,-5) *+!LD{\alpha_{k}} *\cir<2pt>{} ="D",
	(15,-10) *\cir<2pt>{} ="E",
	(6,-10) *\cir<2pt>{} ="F",
	(0,-10) *+!U{\alpha_{2k-1}} *\cir<2pt>{} ="G",
	\ar@{-} "A";"B"
	\ar@{-} "B";(9,0) 
	\ar@{.} (9,0);(12,0)
	\ar@{-} (12,0);"C"
	\ar@{-} "C";"D"
	\ar@{-} "D";"E"
	\ar@{-} "E";(12,-10)
	\ar@{.} (12,-10);(9,-10)
	\ar@{-} (9,-10);"F"
	\ar@{-} "F";"G"
	\ar@{<->} (0,-2);(0,-8)
	\ar@{<->} (6,-2);(6,-8)
	\ar@{<->} (15,-2);(15,-8)
\end{xy}
} 
& $\left\{ 
\raisebox{0.5cm}{
\begin{xy}
	*+!D{b_1} *\cir<2pt>{} ="A",
	(6,0) *+!D{b_2} *\cir<2pt>{} ="B",
	(15,0)  *+!LD{b_{k-1}} *\cir<2pt>{} ="C",
	(21,-5) *+!LD{b_{k}} *\cir<2pt>{} ="D",
	(15,-10) *+!LU{b_{k-1}} *\cir<2pt>{} ="E",
	(6,-10) *+!U{b_{2}} *\cir<2pt>{} ="F",
	(0,-10) *+!U{b_1} *\cir<2pt>{} ="G",
	\ar@{-} "A";"B"
	\ar@{-} "B";(9,0) 
	\ar@{.} (9,0);(12,0)^*U{\cdots}
	\ar@{-} (12,0);"C"
	\ar@{-} "C";"D"
	\ar@{-} "D";"E"
	\ar@{-} "E";(12,-10)
	\ar@{.} (12,-10);(9,-10)
	\ar@{-} (9,-10);"F"
	\ar@{-} "F";"G"
\end{xy}}
\right\}$
\\
\bottomrule
	\end{longtable}
\end{center}

Therefore for each $\mathfrak{g}$,
we can find a basis of $\Psi(\mathfrak{b})$ 
by taking some weighted Dynkin diagrams 
in $\Psi(\mathcal{H}^n(\mathfrak{a}_+)) = 
\{ \Psi_A \in \Psi(\mathcal{H}^n(\mathfrak{j}_+)) \mid \text{$\Psi_A$ matches $S_\mathfrak{g}$}\}$ as follows:

\begin{center}
\begin{longtable}{ll} \toprule
$\mathfrak{g}$ & Example of basis of $\Psi(\mathfrak{b})$ \\ \midrule
$\mathfrak{sl}(2k,\mathbb{R})$ $(2k=l+1)$ & $[3,1^{2k-3}],[5,1^{2k-5}],\dots,[2k-1,1], [2k]$ \\
$\mathfrak{sl}(2k+1,\mathbb{R})$ $(2k=l)$ & $[3,1^{2k-2}],[5,1^{2k-4}],\dots,[2k+1]$ \\
$\mathfrak{su}^*(4m)$ $(4m=l+1)$ & $[3^2,1^{4m-6}],[5^2,1^{4m-10}],\dots,[(2m-1)^2,1^2], [(2m)^2]$ \\
$\mathfrak{su}^*(4m+2)$ $(4m=l-1)$& $[3^2,1^{4m-4}],[5^2,1^{4m-8}],\dots,[(2m+1)^2]$ \\
$\mathfrak{su}(p,q) \ (p+q=l+1,p>q+1)$ & $[3,1^{l-2}],[5,1^{l-4}],\dots,[2q+1,1^{l-2q}]$ \\
$\mathfrak{su}(k+1,k)$ $(2k=l)$ & $[3,1^{2k-2}],[5,1^{2k-4}],\dots,[2k-1,1^2],[2k+1]$ \\
$\mathfrak{su}(k,k)$ $(2k=l+1)$ & $[3,1^{2k-3}],[5,1^{2k-5}],\dots,[2k-1,1],[2k]$ \\
\bottomrule
\end{longtable}
\end{center}

\subsection{Type $B_l$}

Let us consider the case where 
$\mathfrak{g}_\mathbb{C}$ is of type $B_l$ for $l \geq 1$,
that is, 
$\mathfrak{g}_\mathbb{C} \simeq \mathfrak{so}(2l+1,\mathbb{C})$.
Then we have
\[
\Map(\Pi,\mathbb{R})^\iota =
\Map(\Pi,\mathbb{R}) 
= \left\{
\begin{xy}
	*++!D{a_1} *\cir<2pt>{}        ="A",
	(10,0) *++!D{a_2} *\cir<2pt>{} ="B",
	(25,0) *++!D{a_{l-1}} *\cir<2pt>{} ="C",
	(35,0) *++!D{a_{l}} *\cir<2pt>{} ="D",
	\ar@{-} "A";"B"
	\ar@{-} "B"; (15,0)
	\ar@{.} (15,0) ; (20,0)^*U{\cdots}
	\ar@{-} (20,0) ; "C"
	\ar@{=>} "C";"D"
\end{xy} 
\right\}
\]

By \cite[Section 5.3]{Collingwood-McGovern93}, 
we can find some examples of weighted Dynkin diagrams in 
$\Psi(\mathcal{H}^n(\mathfrak{j}_+))$ as follows:
\begin{center}
\begin{longtable}{ll} \toprule
	Symbol & Weighted Dynkin diagram in $\Psi(\mathcal{H}^n(\mathfrak{j}_+))$ \\ \midrule
	
	$[2l+1]$ & \begin{xy}
	*++!D{2} *\cir<2pt>{}        ="A",
	(10,0) *++!D{2} *\cir<2pt>{} ="B",
	(25,0) *++!D{2} *\cir<2pt>{} ="C",
	(35,0) *++!D{2} *\cir<2pt>{} ="D",
	\ar@{-} "A";"B"
	\ar@{-} "B"; (15,0)
	\ar@{.} (15,0) ; (20,0)^*U{\cdots}
	\ar@{-} (20,0) ; "C"
	\ar@{=>} "C";"D"
\end{xy} \\ 

    \raisebox{-0.3cm}{$[2s+1,1^{2l-2s}]$}& \begin{xy}
	*++!D{2} *\cir<2pt>{}        ="A",
	(5,0) *++!D{2} *\cir<2pt>{} ="B",
	(15,0) *++!D{2} *+!U{\alpha_s} *\cir<2pt>{} ="C",
	(20,0) *++!D{0} *\cir<2pt>{} ="D",
	(30,0) *++!D{0} *\cir<2pt>{} ="E",	
	(35,0) *++!D{0} *\cir<2pt>{} ="F",
	\ar@{-} "A";"B"
	\ar@{-} "B"; (8,0)
	\ar@{.} (8,0) ; (12,0)^*U{\cdots}
	\ar@{-} (12,0) ; "C"
	\ar@{-} "C";"D"
	\ar@{-} "D"; (23,0)
	\ar@{.} (23,0) ; (27,0)^*U{\cdots}
	\ar@{-} (27,0) ; "E"
	\ar@{=>} "E";"F"
\end{xy} \\ 
\bottomrule
\end{longtable}
\end{center}

Let $\mathfrak{g}$ be a non-compact real form of $\mathfrak{g}_\mathbb{C}$.
Then the Satake diagram $S_\mathfrak{g}$ and $\Psi(\mathfrak{b})$ are given as follows:

\begin{center}
	\begin{longtable}{lll} \toprule
	$\mathfrak{g}$ & $S_\mathfrak{g}$ & $\Psi(\mathfrak{b})$ \\ \midrule
	
\raisebox{-0.5cm}{\shortstack{$\mathfrak{so}(p,q)$ \\ $(p+q=2l+1)$ \\ $(p>q+1)$}} & 
\begin{xy}
	*\cir<2pt>{}        ="A",
	(5,0) *\cir<2pt>{} ="B",
	(15,0) *+!U{\alpha_q} *\cir<2pt>{} ="C",
	(20,0) *{\bullet} ="D",
	(30,0) *{\bullet} ="E",	
	(35,0)  *{\bullet} ="F",
	\ar@{-} "A";"B"
	\ar@{-} "B"; (8,0)
	\ar@{.} (8,0) ; (12,0)
	\ar@{-} (12,0) ; "C"
	\ar@{-} "C";"D"
	\ar@{-} "D"; (23,0)
	\ar@{.} (23,0) ; (27,0)
	\ar@{-} (27,0) ; "E"
	\ar@{=>} "E";"F"
\end{xy} & 
$\left\{
\begin{xy}
	*++!D{b_1} *\cir<2pt>{}        ="A",
	(5,0) *++!D{b_2} *\cir<2pt>{} ="B",
	(15,0) *++!D{b_q} *\cir<2pt>{} ="C",
	(20,0) *++!D{0} *\cir<2pt>{} ="D",
	(30,0) *++!D{0} *\cir<2pt>{} ="E",	
	(35,0)  *++!D{0} *\cir<2pt>{} ="F",
	\ar@{-} "A";"B"
	\ar@{-} "B"; (8,0)
	\ar@{.} (8,0) ; (12,0)
	\ar@{-} (12,0) ; "C"
	\ar@{-} "C";"D"
	\ar@{-} "D"; (23,0)
	\ar@{.} (23,0) ; (27,0)
	\ar@{-} (27,0) ; "E"
	\ar@{=>} "E";"F"
\end{xy}
\right\}$
\\ \hline
	
$\mathfrak{so}(l+1,l)$ & \begin{xy}
    *\cir<2pt>{}        ="A",
	(10,0) *\cir<2pt>{} ="B",
	(25,0) *\cir<2pt>{} ="C",
	(35,0) *\cir<2pt>{} ="D",
	\ar@{-} "A";"B"
	\ar@{-} "B"; (15,0)
	\ar@{.} (15,0) ; (20,0)
	\ar@{-} (20,0) ; "C"
	\ar@{=>} "C";"D"
\end{xy} &
$\left\{
\begin{xy}
	*++!D{b_1} *\cir<2pt>{}        ="A",
	(10,0) *++!D{b_2} *\cir<2pt>{} ="B",
	(25,0) *++!D{b_{l-1}} *\cir<2pt>{} ="C",
	(35,0) *++!D{b_{l}} *\cir<2pt>{} ="D",
	\ar@{-} "A";"B"
	\ar@{-} "B"; (15,0)
	\ar@{.} (15,0) ; (20,0)^*U{\cdots}
	\ar@{-} (20,0) ; "C"
	\ar@{=>} "C";"D"
\end{xy} 
\right\}$
\\ 
\bottomrule
	\end{longtable}
\end{center}

Therefore for each $\mathfrak{g}$,
we can find a basis of $\Psi(\mathfrak{b})$ 
by taking some weighted Dynkin diagrams 
in $\Psi(\mathcal{H}^n(\mathfrak{a}_+)) = 
\{ \Psi_A \in \Psi(\mathcal{H}^n(\mathfrak{j}_+)) \mid \text{$\Psi_A$ matches $S_\mathfrak{g}$}\}$ as follows:

\begin{center}
\begin{longtable}{ll} \toprule
$\mathfrak{g}$ & Example of basis of $\Psi(\mathfrak{b})$ \\ \midrule
$\mathfrak{so}(p,q) \ (p+q=2l+1,p>q+1)$ & $[3,1^{2l-2}],[5,1^{2l-4}],\dots,[2q+1,1^{2l-2q}]$ \\
$\mathfrak{so}(l+1,l)$ & $[3,1^{2l-2}],[5,1^{2l-4}],\dots,[2l+1]$ \\
\bottomrule
\end{longtable}
\end{center}

\subsection{Type $C_l$}

Let us consider the case where 
$\mathfrak{g}_\mathbb{C}$ is of type $C_l$ for $l \geq 1$,
that is, 
$\mathfrak{g}_\mathbb{C} \simeq \mathfrak{sp}(l,\mathbb{C})$.
Then we have
\[
\Map(\Pi,\mathbb{R})^\iota = \Map(\Pi,\mathbb{R})
= \left\{
\begin{xy}
        *++!D{a_1} *\cir<2pt>{}        ="A",
	(10,0) *++!D{a_2} *\cir<2pt>{} ="B",
	(25,0) *++!D{a_{l-1}} *\cir<2pt>{} ="C",
	(35,0) *++!D{a_l} *\cir<2pt>{} ="D",
	\ar@{-} "A";"B"
	\ar@{-} "B"; (15,0)
	\ar@{.} (15,0) ; (20,0)^*U{\cdots}
	\ar@{-} (20,0) ; "C"
	\ar@{<=} "C";"D"
\end{xy}
\right\}
\]

By \cite[Section 5.3]{Collingwood-McGovern93}, 
we can find some examples of weighted Dynkin diagrams in 
$\Psi(\mathcal{H}^n(\mathfrak{j}_+))$ as follows:
\begin{center}
\begin{longtable}{ll} \toprule
	Symbol & Weighted Dynkin diagram in $\Psi(\mathcal{H}^n(\mathfrak{j}_+))$ \\ \midrule
	$[2^{l}]$ & \begin{xy}
	*++!D{0} *\cir<2pt>{} ="A",
	(3,0) = "A_1",
	(9,0) = "A_2",
	(12,0) *++!D{0} *\cir<2pt>{} ="B",
	(18,0) *++!D{0} *\cir<2pt>{} ="C",
	(24,0) *++!D{0} *\cir<2pt>{} ="D",
	(27,0) = "D_1",
	(33,0) = "D_2",
	(36,0) *++!D{0} *\cir<2pt>{} ="E",
	(42,0) *++!D{2} *\cir<2pt>{} ="F",	
	\ar@{-} "A";"A_1"
	\ar@{.} "A_1";"A_2"^*U{\cdots}
	\ar@{-} "A_2";"B"
	\ar@{-} "B";"C"
	\ar@{-} "C";"D"
	\ar@{-} "D";"D_1"
	\ar@{.} "D_1";"D_2"^*U{\cdots}
	\ar@{-} "D_2";"E"
	\ar@{<=} "E";"F"
\end{xy} \\
	\raisebox{-0.3cm}{\shortstack{$[2s+2,2^{l-s}]$}} & \begin{xy}
	*++!D{2} *\cir<2pt>{} ="A",
	(3,0) = "A_1",
	(9,0) = "A_2",
	(12,0) *++!D{2} *\cir<2pt>{} ="B",
	(18,0) *++!D{2} *++!U{\alpha_s} *\cir<2pt>{} ="C",
	(24,0) *++!D{0} *\cir<2pt>{} ="D",
	(27,0) = "D_1",
	(33,0) = "D_2",
	(36,0) *++!D{0} *\cir<2pt>{} ="E",
	(42,0) *++!D{2} *\cir<2pt>{} ="F",	
	\ar@{-} "A";"A_1"
	\ar@{.} "A_1";"A_2"^*U{\cdots}
	\ar@{-} "A_2";"B"
	\ar@{-} "B";"C"
	\ar@{-} "C";"D"
	\ar@{-} "D";"D_1"
	\ar@{.} "D_1";"D_2"^*U{\cdots}
	\ar@{-} "D_2";"E"
	\ar@{<=} "E";"F"
\end{xy} \\ 
$[(2s+1)^2,1^{2l-4s-2}]$ & \begin{xy}
	(0,0) *++!D{0} *\cir<2pt>{}        ="AAAA",
	(3,0) *++!D{2} *\cir<2pt>{}        ="AAA",
	(6,0)*++!D{0} *\cir<2pt>{}        ="AA",
	(9,0)*++!D{2} *\cir<2pt>{}        ="A",
	(12,0) = "A_1",
	(17,0) = "A_2",
	(20,0) *++!D{0} *\cir<2pt>{} ="B",
	(23,0) *++!D{2} *++!U{\alpha_{2s}} *\cir<2pt>{} ="C",
	(26,0) *++!D{0} *\cir<2pt>{} ="DD",
	(29,0) *++!D{0} *\cir<2pt>{} ="D",
	(32,0) = "D_1",
	(37,0) = "D_2",
	(40,0) *++!D{0} *\cir<2pt>{} ="E",
	(43,0) *++!D{0} *\cir<2pt>{} ="EE",
	(48,0) *++!D{0} *\cir<2pt>{} ="F",
	\ar@{-} "AAAA";"AAA"
	\ar@{-} "AAA";"AA"
	\ar@{-} "AA";"A"
	\ar@{-} "A";"A_1"
	\ar@{.} "A_1";"A_2"^*!U{\cdots}
	\ar@{-} "A_2";"B"
	\ar@{-} "B";"C"
	\ar@{-} "C";"DD"
	\ar@{-} "DD";"D"
	\ar@{-} "D";"D_1"
	\ar@{.} "D_1";"D_2"^*!U{\cdots}
	\ar@{-} "D_2";"E"
	\ar@{-} "E";"EE"
	\ar@{<=} "EE";"F"
\end{xy} \\ 	
\raisebox{-0.3cm}{\shortstack{$[(2k)^2]$ \\ $(2k=l)$}}
 & \begin{xy}
	(0,0) *++!D{0} *\cir<2pt>{}        ="AAAA",
	(3,0) *++!D{2} *\cir<2pt>{}        ="AAA",
	(6,0)*++!D{0} *\cir<2pt>{}        ="AA",
	(9,0)*++!D{2} *\cir<2pt>{}        ="A",
	(12,0) = "A_1",
	(17,0) = "A_2",
	(20,0) *++!D{0} *\cir<2pt>{} ="B",
	(23,0) *++!D{2} *\cir<2pt>{} ="C",
	(26,0) *++!D{0} *\cir<2pt>{} ="DD",
	(29,0) *++!D{2} *\cir<2pt>{} ="D",
	(32,0) = "D_1",
	(37,0) = "D_2",
	(40,0) *++!D{2} *\cir<2pt>{} ="E",
	(43,0) *++!D{0} *\cir<2pt>{} ="EE",
	(48,0) *++!D{2} *\cir<2pt>{} ="F",
	\ar@{-} "AAAA";"AAA"
	\ar@{-} "AAA";"AA"
	\ar@{-} "AA";"A"
	\ar@{-} "A";"A_1"
	\ar@{.} "A_1";"A_2"^*!U{\cdots}
	\ar@{-} "A_2";"B"
	\ar@{-} "B";"C"
	\ar@{-} "C";"DD"
	\ar@{-} "DD";"D"
	\ar@{-} "D";"D_1"
	\ar@{.} "D_1";"D_2"^*!U{\cdots}
	\ar@{-} "D_2";"E"
	\ar@{-} "E";"EE"
	\ar@{<=} "EE";"F"
\end{xy} \\ 
\bottomrule
\end{longtable}
\end{center}

Let $\mathfrak{g}$ be a non-compact real form of $\mathfrak{g}_\mathbb{C}$.
Then the Satake diagram $S_\mathfrak{g}$ and $\Psi(\mathfrak{b})$ are given as follows:
\begin{center}
	\begin{longtable}{lll} \toprule
	$\mathfrak{g}$ & $S_\mathfrak{g}$ & $\Psi(\mathfrak{b})$ \\ \midrule

$\mathfrak{sp}(l,\mathbb{R})$ & 
\begin{xy}
    *\cir<2pt>{}        ="A",
	(6,0) *\cir<2pt>{} ="B",
	(15,0) *\cir<2pt>{} ="C",
	(21,0) *\cir<2pt>{} ="D",
	\ar@{-} "A";"B"
	\ar@{-} "B"; (9,0)
	\ar@{.} (9,0) ; (12,0)
	\ar@{-} (12,0) ; "C"
	\ar@{<=} "C";"D"
\end{xy} &
$\left\{
\begin{xy}
    *++!D{b_1} *\cir<2pt>{}        ="A",
	(6,0) *++!D{b_2} *\cir<2pt>{} ="B",
	(15,0) *++!D{b_{l-1}} *\cir<2pt>{} ="C",
	(21,0) *++!D{b_{l}} *\cir<2pt>{} ="D",
	\ar@{-} "A";"B"
	\ar@{-} "B"; (9,0)
	\ar@{.} (9,0) ; (12,0)
	\ar@{-} (12,0) ; "C"
	\ar@{<=} "C";"D"
\end{xy}
\right\}$
\\ \hline
	
\raisebox{-0.5cm}{\shortstack{$\mathfrak{sp}(p,q)$ \\ $(p+q=l)$ \\ $(p>q)$}}& 
\begin{xy}
    *{\bullet}       ="A",
	(4,0) *\cir<2pt>{} ="B",
	(8,0) *{\bullet} ="C",
	(15,0) *+!U{\alpha_{2q}} *\cir<2pt>{} ="D",
	(19,0) *{\bullet} ="E",
	(23,0) *{\bullet} ="E'",
	(30,0) *{\bullet} ="F",	
	(34,0) *{\bullet} ="G",
	\ar@{-} "A";"B"
	\ar@{-} "B";"C"	
	\ar@{-} "C"; (10,0)
	\ar@{.} (10,0) ; (13,0)
	\ar@{-} (13,0) ; "D"
	\ar@{-} "D";"E"
	\ar@{-} "E";"E'"
	\ar@{-} "E'"; (25,0)
	\ar@{.} (25,0) ; (28,0)
	\ar@{-} (28,0) ; "F"
	\ar@{<=} "F";"G"
\end{xy} &
$\left\{
\begin{xy}
        *++!D{0} *\cir<2pt>{}      ="A",
	(4,0) *++!D{b_1} *\cir<2pt>{} ="B",
	(7,0) *++!D{0} *\cir<2pt>{} ="C",
	(16,0) *++!D{b_{q}} *\cir<2pt>{} ="D",
	(19,0) *++!D{0} *\cir<2pt>{} ="E",
	(23,0) *++!D{0} *\cir<2pt>{} ="E'",
	(30,0) *++!D{0} *\cir<2pt>{} ="F",	
	(34,0) *++!D{0} *\cir<2pt>{} ="G",
	\ar@{-} "A";"B"
	\ar@{-} "B";"C"	
	\ar@{-} "C"; (9,0)
	\ar@{.} (9,0) ; (14,0)^*U{\cdots}
	\ar@{-} (14,0) ; "D"
	\ar@{-} "D";"E"
	\ar@{-} "E";"E'"
	\ar@{-} "E'"; (25,0)
	\ar@{.} (25,0) ; (28,0)^*U{\cdots}
	\ar@{-} (28,0) ; "F"
	\ar@{<=} "F";"G"
\end{xy}
\right\}$
\\ \hline

\raisebox{-0.3cm}{\shortstack{$\mathfrak{sp}(k,k)$ \\ $(2k=l)$}}& 
\begin{xy}
	*{\bullet}       ="A",
	(5,0) *\cir<2pt>{} ="B",
	(10,0) *{\bullet} ="C",
	(20,0) *\cir<2pt>{} ="D",
	(25,0) *{\bullet} ="E",	
	(30,0) *+!LU{\alpha_{2k}} *\cir<2pt>{} ="F",
	\ar@{-} "A";"B"
	\ar@{-} "B";"C"	
	\ar@{-} "C"; (13,0)
	\ar@{.} (13,0) ; (17,0)
	\ar@{-} (17,0) ; "D"
	\ar@{-} "D";"E"
	\ar@{<=} "E";"F"
\end{xy} &
$\left\{
\begin{xy}
	*++!D{0} *\cir<2pt>{}       ="A",
	(5,0) *++!D{b_1} *\cir<2pt>{} ="B",
	(9,0) *++!D{0} *\cir<2pt>{} ="C",
	(20,0) *++!D{b_{k-1}} *\cir<2pt>{} ="D",
	(25,0) *++!D{0} *\cir<2pt>{} ="E",	
	(30,0) *++!D{b_k} *\cir<2pt>{} ="F",
	\ar@{-} "A";"B"
	\ar@{-} "B";"C"	
	\ar@{-} "C"; (11,0)
	\ar@{.} (11,0) ; (15,0)^*U{\cdots}
	\ar@{-} (15,0) ; "D"
	\ar@{-} "D";"E"
	\ar@{<=} "E";"F"
\end{xy}
\right\}$ 
\\ 
\bottomrule
	\end{longtable}
\end{center}

Therefore for each $\mathfrak{g}$,
we can find a basis of $\Psi(\mathfrak{b})$ 
by taking some weighted Dynkin diagrams 
in $\Psi(\mathcal{H}^n(\mathfrak{a}_+)) = 
\{ \Psi_A \in \Psi(\mathcal{H}^n(\mathfrak{j}_+)) \mid \text{$\Psi_A$ matches $S_\mathfrak{g}$}\}$ as follows:
\begin{center}
\begin{longtable}{ll} \toprule
$\mathfrak{g}$ & Example of basis of $\Psi(\mathfrak{b})$ \\ \midrule
$\mathfrak{sp}(l,\mathbb{R})$ & $[2^{l}],[4,2^{l-2}], [6,2^{l-3}], \dots, [2l]$ \\
$\mathfrak{sp}(p,q) \ (p+q =l,p>q)$ & $[3^2,1^{2l-6}],[5^2,1^{2l-10}],\dots,[(2q+1)^2,1^{2l-2q-1}]$ \\
$\mathfrak{sp}(k,k)$ $(2k=l)$ & $[3^2,1^{4k-6}],[5^2,1^{4k-10}],\dots,[(2k-1)^2,1^{2}],[(2k)^2]$ \\
\bottomrule
\end{longtable}
\end{center}

\subsection{Type $D_{2m}$}

Let us consider the case where 
$\mathfrak{g}_\mathbb{C}$ is of type $D_{2m}$ for $m \geq 2$,
that is,
$\mathfrak{g}_\mathbb{C} \simeq \mathfrak{so}(4m,\mathbb{C})$.
Then we have 
\[
\Map(\Pi,\mathbb{R})^\iota = \Map(\Pi,\mathbb{R}) =
\left\{ 
\begin{xy}
	*++!D{a_1} *\cir<2pt>{}        ="A",
	(10,0) *++!D{a_2} *\cir<2pt>{} ="B",
	(25,0) *++!D{a_{2m-2}} *\cir<2pt>{} ="C",
	(35,-5) *++!D{a_{2m}} *++!U{} *\cir<2pt>{} ="D",
	(35,5) *++!D{a_{2m-1}} *\cir<2pt>{} ="E",	
	\ar@{-} "A";"B"
	\ar@{-} "B"; (15,0)
	\ar@{.} (15,0) ; (20,0)^*!U{\cdots}
	\ar@{-} (20,0) ; "C"
	\ar@{-} "C";"D"
	\ar@{-} "C";"E"
\end{xy}
\right\}
\]

By \cite[Section 5.3]{Collingwood-McGovern93}, 
we can find some examples of weighted Dynkin diagrams in 
$\Psi(\mathcal{H}^n(\mathfrak{j}_+))$ as follows:
\begin{center}
\begin{longtable}{ll} \toprule
	Symbol & Weighted Dynkin diagram in $\Psi(\mathcal{H}^n(\mathfrak{j}_+))$ \\ \midrule
	
	$[2s+1,1^{4m-2s-1}]$ & \begin{xy}
	*++!D{2} *\cir<2pt>{}        ="A",
	(5,0) *++!D{2} *\cir<2pt>{} ="B",
	(15,0) *++!D{2} *+!U{\alpha_s} *\cir<2pt>{} ="C",
	(20,0) *++!D{0} *\cir<2pt>{} ="D",
	(30,0) *++!D{0} *\cir<2pt>{} ="E",	
	(35,-5) *++!L{0} *\cir<2pt>{} ="F",
	(35,5) *++!L{0} *\cir<2pt>{} ="G",
	\ar@{-} "A";"B"
	\ar@{-} "B"; (8,0)
	\ar@{.} (8,0) ; (12,0)
	\ar@{-} (12,0) ; "C"
	\ar@{-} "C";"D"
	\ar@{-} "D"; (23,0)
	\ar@{.} (23,0) ; (27,0)
	\ar@{-} (27,0) ; "E"
	\ar@{-} "E";"F"
	\ar@{-} "E";"G"	
\end{xy}  \\ 
	$[4m-1,1]$ & \begin{xy}
	*++!D{2} *\cir<2pt>{}        ="A",
	(5,0) *++!D{2} *\cir<2pt>{} ="B",
	(15,0) *++!D{2} *\cir<2pt>{} ="C",
	(20,0) *++!D{2} *\cir<2pt>{} ="D",
	(30,0) *++!D{2} *\cir<2pt>{} ="E",	
	(35,-5) *++!L{2} *\cir<2pt>{} ="F",
	(35,5) *++!L{2} *\cir<2pt>{} ="G",
	\ar@{-} "A";"B"
	\ar@{-} "B"; (8,0)
	\ar@{.} (8,0) ; (12,0)
	\ar@{-} (12,0) ; "C"
	\ar@{-} "C";"D"
	\ar@{-} "D"; (23,0)
	\ar@{.} (23,0) ; (27,0)
	\ar@{-} (27,0) ; "E"
	\ar@{-} "E";"F"
	\ar@{-} "E";"G"	
\end{xy}  \\
$[2^{2m}]_{\text{I}}$ & \begin{xy}
	*++!D{0} *\cir<2pt>{}        ="A",
	(5,0) *++!D{0} *\cir<2pt>{} ="B",
	(15,0) *++!D{0} *\cir<2pt>{} ="C",
	(20,0) *++!D{0} *\cir<2pt>{} ="D",
	(30,0) *++!D{0} *\cir<2pt>{} ="E",	
	(35,-5) *++!L{2} *\cir<2pt>{} ="F",
	(35,5) *++!L{0} *\cir<2pt>{} ="G",
	\ar@{-} "A";"B"
	\ar@{-} "B"; (8,0)
	\ar@{.} (8,0) ; (12,0)
	\ar@{-} (12,0) ; "C"
	\ar@{-} "C";"D"
	\ar@{-} "D"; (23,0)
	\ar@{.} (23,0) ; (27,0)
	\ar@{-} (27,0) ; "E"
	\ar@{-} "E";"F"
	\ar@{-} "E";"G"	
\end{xy}  \\ 
$[(2s+1)^2,1^{4m-4s-2}]$ & \begin{xy}
	(0,0) *++!D{0} *\cir<2pt>{}        ="AAAA",
	(3,0) *++!D{2} *\cir<2pt>{}        ="AAA",
	(6,0)*++!D{0} *\cir<2pt>{}        ="AA",
	(9,0)*++!D{2} *\cir<2pt>{}        ="A",
	(12,0) = "A_1",
	(17,0) = "A_2",
	(20,0) *++!D{0} *\cir<2pt>{} ="B",
	(23,0) *++!D{2} *++!U{\alpha_{2s}} *\cir<2pt>{} ="C",
	(26,0) *++!D{0} *\cir<2pt>{} ="DD",
	(29,0) *++!D{0} *\cir<2pt>{} ="D",
	(32,0) = "D_1",
	(37,0) = "D_2",
	(40,0) *++!D{0} *\cir<2pt>{} ="E",
	(43,5) *++!D{0} *\cir<2pt>{} ="F",
	(43,-5) *++!D{0} *\cir<2pt>{} ="F'",
	\ar@{-} "AAAA";"AAA"
	\ar@{-} "AAA";"AA"
	\ar@{-} "AA";"A"
	\ar@{-} "A";"A_1"
	\ar@{.} "A_1";"A_2"^*!U{\cdots}
	\ar@{-} "A_2";"B"
	\ar@{-} "B";"C"
	\ar@{-} "C";"DD"
	\ar@{-} "DD";"D"
	\ar@{-} "D";"D_1"
	\ar@{.} "D_1";"D_2"^*!U{\cdots}
	\ar@{-} "D_2";"E"
	\ar@{-} "E";"F"
	\ar@{-} "E";"F'"
\end{xy} \\ 
\bottomrule
\end{longtable}
\end{center}

Let $\mathfrak{g}$ be a non-compact real form of $\mathfrak{g}_\mathbb{C}$.
Then the Satake diagram $S_\mathfrak{g}$ and $\Psi(\mathfrak{b})$ are given as follows:

\begin{center}
	\begin{longtable}{lll} \toprule
	$\mathfrak{g}$ & $S_\mathfrak{g}$ & $\Psi(\mathfrak{b})$ \\ \midrule

\raisebox{-0.5cm}{\shortstack{$\mathfrak{so}(p,q)$ \\ $(p+q=4m)$ \\ $(p>q+2)$}}& 
\begin{xy}
	*\cir<2pt>{}        ="A",
	(5,0) *\cir<2pt>{} ="B",
	(15,0) *+!U{\alpha_q} *\cir<2pt>{} ="C",
	(20,0) *{\bullet} ="D",
	(30,0) *{\bullet} ="E",	
	(35,-5) *++!U{} *{\bullet} ="F",
	(35,5) *++!D{} *{\bullet} ="G",
	\ar@{-} "A";"B"
	\ar@{-} "B"; (8,0)
	\ar@{.} (8,0) ; (12,0)
	\ar@{-} (12,0) ; "C"
	\ar@{-} "C";"D"
	\ar@{-} "D"; (23,0)
	\ar@{.} (23,0) ; (27,0)
	\ar@{-} (27,0) ; "E"
	\ar@{-} "E";"F"
	\ar@{-} "E";"G"	
\end{xy} 
&
$\left\{
\begin{xy}
	*++!D{b_1} *\cir<2pt>{}        ="A",
	(5,0) *++!D{b_2} *\cir<2pt>{} ="B",
	(15,0) *++!D{b_q}  *\cir<2pt>{} ="C",
	(20,0) *++!D{0} *\cir<2pt>{} ="D",
	(30,0) *++!D{0} *\cir<2pt>{} ="E",	
	(35,-5) *++!U{0} *\cir<2pt>{} ="F",
	(35,5) *++!D{0} *\cir<2pt>{} ="G",
	\ar@{-} "A";"B"
	\ar@{-} "B"; (8,0)
	\ar@{.} (8,0) ; (12,0)
	\ar@{-} (12,0) ; "C"
	\ar@{-} "C";"D"
	\ar@{-} "D"; (23,0)
	\ar@{.} (23,0) ; (27,0)
	\ar@{-} (27,0) ; "E"
	\ar@{-} "E";"F"
	\ar@{-} "E";"G"	
\end{xy}
\right\}$
\\ \hline

$\mathfrak{so}(2m+1,2m-1)$ & \begin{xy}
	*\cir<2pt>{}       ="A",
	(5,0) *\cir<2pt>{} ="B",
	(10,0) *\cir<2pt>{} ="C",
	(20,0) *\cir<2pt>{} ="D",
	(25,0) *\cir<2pt>{} ="E",
	(30,0) *\cir<2pt>{} ="F",	
	(35,-5) *++!U{} *\cir<2pt>{} ="G",
	(35,5) *++!D{} *\cir<2pt>{} ="H",	
	\ar@{-} "A";"B"
	\ar@{-} "B";"C"	
	\ar@{-} "C"; (13,0)
	\ar@{.} (13,0) ; (17,0)
	\ar@{-} (17,0) ; "D"
	\ar@{-} "D";"E"
	\ar@{-} "E";"F"
	\ar@{-} "F";"G"
	\ar@{-} "F";"H"
	\ar@/_2mm/@{<->}@<-1mm> "G";"H"
\end{xy} &
$\left\{
\begin{xy}
	*++!D{b_1} *\cir<2pt>{}       ="A",
	(8,0) *++!D{b_2} *\cir<2pt>{} ="B",
	(10,0) ="C",
	(14,0) ="C'",
	(18,0) = "C''",
	(22,0) *++!D{b_{2m-3}} *\cir<2pt>{} ="E",
	(30,0) *++!L{b_{2m-2}} *\cir<2pt>{} ="F",	
	(35,-5) *++!U{b_{2m-1}} *\cir<2pt>{} ="G",
	(35,5) *++!D{b_{2m-1}} *\cir<2pt>{} ="H",	
	\ar@{-} "A";"B"
	\ar@{-} "B";"C"	
	\ar@{-} "C";"C'"
	\ar@{.} "C'"; "C''"
	\ar@{-} "C''" ; "E"
	\ar@{-} "E";"F"
	\ar@{-} "F";"G"
	\ar@{-} "F";"H"
\end{xy}
\right\}$ \\ \hline

$\mathfrak{so}(2m,2m)$ & \begin{xy}
	*\cir<2pt>{}       ="A",
	(5,0) *\cir<2pt>{} ="B",
	(10,0) *\cir<2pt>{} ="C",
	(20,0) *\cir<2pt>{} ="D",
	(25,0) *\cir<2pt>{} ="E",
	(30,0) *\cir<2pt>{} ="F",	
	(35,-5) *++!U{} *\cir<2pt>{} ="G",
	(35,5) *++!D{} *\cir<2pt>{} ="H",	
	\ar@{-} "A";"B"
	\ar@{-} "B";"C"	
	\ar@{-} "C"; (13,0)
	\ar@{.} (13,0) ; (17,0)
	\ar@{-} (17,0) ; "D"
	\ar@{-} "D";"E"
	\ar@{-} "E";"F"
	\ar@{-} "F";"G"
	\ar@{-} "F";"H"
\end{xy} &
$\left\{
\begin{xy}
	*++!D{b_1} *\cir<2pt>{}       ="A",
	(8,0) *++!D{b_2} *\cir<2pt>{} ="B",
	(12,0) ="C",
	(16,0) ="C'",
	(22,0) *++!D{b_{2m-3}} *\cir<2pt>{} ="E",
	(30,0) *++!L{b_{2m-2}} *\cir<2pt>{} ="F",	
	(35,-5) *++!U{b_{2m}} *\cir<2pt>{} ="G",
	(35,5) *++!D{b_{2m-1}} *\cir<2pt>{} ="H",	
	\ar@{-} "A";"B"
	\ar@{-} "B";"C"	
	\ar@{.} "C";"C'"^*U{\cdots}
	\ar@{-} "C'"; "E"
	\ar@{-} "E";"F"
	\ar@{-} "F";"G"
	\ar@{-} "F";"H"
\end{xy}
\right\}$ \\ \hline

$\mathfrak{so}^*(4m)$ & \begin{xy}
	*{\bullet}       ="A",
	(5,0) *\cir<2pt>{} ="B",
	(10,0) *{\bullet} ="C",
	(20,0) *\cir<2pt>{} ="D",
	(25,0)  *{\bullet}="E",
	(30,0)  *\cir<2pt>{} ="F",	
	(35,-5) *++!U{} *\cir<2pt>{} ="G",
	(35,5) *++!D{} *{\bullet} ="H",	
	\ar@{-} "A";"B"
	\ar@{-} "B";"C"	
	\ar@{-} "C"; (13,0)
	\ar@{.} (13,0) ; (17,0)
	\ar@{-} (17,0) ; "D"
	\ar@{-} "D";"E"
	\ar@{-} "E";"F"
	\ar@{-} "F";"G"
	\ar@{-} "F";"H"
\end{xy} &
$\left\{
\begin{xy}
	*++!D{0} *\cir<2pt>{}       ="A",
	(5,0) *++!D{b_1} *\cir<2pt>{} ="B",
	(10,0) *++!D{0} *\cir<2pt>{} ="C",
	(21,0) *++!D{b_{m-2}} *\cir<2pt>{} ="D",
	(26,0) *++!D{0} *\cir<2pt>{} ="E",
	(30,0)  *++!L{b_{m-1}} *\cir<2pt>{} ="F",	
	(35,-5) *++!U{b_m} *\cir<2pt>{} ="G",
	(35,5) *++!D{0} *\cir<2pt>{} ="H",	
	\ar@{-} "A";"B"
	\ar@{-} "B";"C"	
	\ar@{-} "C"; (13,0)
	\ar@{.} (13,0) ; (16,0)^*U{\cdots}
	\ar@{-} (16,0) ; "D"
	\ar@{-} "D";"E"
	\ar@{-} "E";"F"
	\ar@{-} "F";"G"
	\ar@{-} "F";"H"
\end{xy}
\right\}$
\\ 
\bottomrule
	\end{longtable}
\end{center}

Therefore for each $\mathfrak{g}$,
we can find a basis of $\Psi(\mathfrak{b})$ 
by taking some weighted Dynkin diagrams 
in $\Psi(\mathcal{H}^n(\mathfrak{a}_+)) = 
\{ \Psi_A \in \Psi(\mathcal{H}^n(\mathfrak{j}_+)) \mid \text{$\Psi_A$ matches $S_\mathfrak{g}$}\}$ as follows:
\begin{center}
\begin{longtable}{ll} \toprule
$\mathfrak{g}$ & Example of basis of $\Psi(\mathfrak{b})$ \\ \midrule
$\mathfrak{so}(p,q) \ (p+q = 4m,p>q+2)$ & $[3,1^{4m-3}],[5,1^{4m-5}],\dots,[2q+1,1^{4m-2q-1}]$\\
$\mathfrak{so}(2m+1,2m-1)$ & $[3,1^{4m-3}],[5,1^{4m-5}],\dots,[4m-1,1]$ \\
$\mathfrak{so}(2m,2m)$ & $[3,1^{4m-1}],[5,1^{4m-3}],\dots,[4m-1,1],[2^{2m}]_{\text{I}}$ \\
$\mathfrak{so}^*(4m)$ & $[3^2,1^{4m-6}],[5^2,1^{4m-10}],\dots,[(2m-1)^2,1^2], [2^{2m}]_{\text{I}}$\\
\bottomrule
\end{longtable}
\end{center}

\subsection{Type $D_{2m+1}$}

Let us consider the case where 
$\mathfrak{g}_\mathbb{C}$ is of type $D_{2m+1}$ for $m \geq 1$,
that is,
$\mathfrak{g}_\mathbb{C} \simeq \mathfrak{so}(4m+2,\mathbb{C})$.
Then we have
\[
\Map(\Pi,\mathbb{R})^\iota 
= \left\{ 
\begin{xy}
	*++!D{a_1} *\cir<2pt>{}        ="A",
	(10,0) *++!D{a_2} *\cir<2pt>{} ="B",
	(25,0) *++!D{a_{2m-1}} *\cir<2pt>{} ="C",
	(35,-5) *++!D{a_{2m+1}} *++!U{} *\cir<2pt>{} ="D",
	(35,5) *++!D{a_{2m}} *\cir<2pt>{} ="E",	
	\ar@{-} "A";"B"
	\ar@{-} "B"; (15,0)
	\ar@{.} (15,0) ; (20,0)^*!U{\cdots}
	\ar@{-} (20,0) ; "C"
	\ar@{-} "C";"D"
	\ar@{-} "C";"E"
\end{xy}
 ~\middle|~
a_{2m} = a_{2m+1}
\right\}
\]

By \cite[Section 5.3]{Collingwood-McGovern93}, 
we can find some examples of weighted Dynkin diagrams in 
$\Psi(\mathcal{H}^n(\mathfrak{j}_+))$ as follows:
\begin{center}
\begin{longtable}{ll} \toprule
	Symbol & Weighted Dynkin diagram in $\Psi(\mathcal{H}^n(\mathfrak{j}_+))$ \\ \midrule
	
	$[2s+1,1^{4m-2s+1}]$ & \begin{xy}
	*++!D{2} *\cir<2pt>{}        ="A",
	(5,0) *++!D{2} *\cir<2pt>{} ="B",
	(15,0) *++!D{2} *+!U{\alpha_s} *\cir<2pt>{} ="C",
	(20,0) *++!D{0} *\cir<2pt>{} ="D",
	(30,0) *++!D{0} *\cir<2pt>{} ="E",	
	(35,-5) *++!L{0} *\cir<2pt>{} ="F",
	(35,5) *++!L{0} *\cir<2pt>{} ="G",
	\ar@{-} "A";"B"
	\ar@{-} "B"; (8,0)
	\ar@{.} (8,0) ; (12,0)
	\ar@{-} (12,0) ; "C"
	\ar@{-} "C";"D"
	\ar@{-} "D"; (23,0)
	\ar@{.} (23,0) ; (27,0)
	\ar@{-} (27,0) ; "E"
	\ar@{-} "E";"F"
	\ar@{-} "E";"G"	
\end{xy}  \\ 
	$[4m+1,1]$ & \begin{xy}
	*++!D{2} *\cir<2pt>{}        ="A",
	(5,0) *++!D{2} *\cir<2pt>{} ="B",
	(15,0) *++!D{2} *\cir<2pt>{} ="C",
	(20,0) *++!D{2} *\cir<2pt>{} ="D",
	(30,0) *++!D{2} *\cir<2pt>{} ="E",	
	(35,-5) *++!L{2} *\cir<2pt>{} ="F",
	(35,5) *++!L{2} *\cir<2pt>{} ="G",
	\ar@{-} "A";"B"
	\ar@{-} "B"; (8,0)
	\ar@{.} (8,0) ; (12,0)
	\ar@{-} (12,0) ; "C"
	\ar@{-} "C";"D"
	\ar@{-} "D"; (23,0)
	\ar@{.} (23,0) ; (27,0)
	\ar@{-} (27,0) ; "E"
	\ar@{-} "E";"F"
	\ar@{-} "E";"G"	
\end{xy}  \\ 
$[(2s+1)^2,1^{4m-4s}]$ & \begin{xy}
	(0,0) *++!D{0} *\cir<2pt>{}        ="AAAA",
	(3,0) *++!D{2} *\cir<2pt>{}        ="AAA",
	(6,0)*++!D{0} *\cir<2pt>{}        ="AA",
	(9,0)*++!D{2} *\cir<2pt>{}        ="A",
	(12,0) = "A_1",
	(17,0) = "A_2",
	(20,0) *++!D{0} *\cir<2pt>{} ="B",
	(23,0) *++!D{2} *++!U{\alpha_{2s}} *\cir<2pt>{} ="C",
	(26,0) *++!D{0} *\cir<2pt>{} ="DD",
	(29,0) *++!D{0} *\cir<2pt>{} ="D",
	(32,0) = "D_1",
	(37,0) = "D_2",
	(40,0) *++!D{0} *\cir<2pt>{} ="E",
	(43,5) *++!D{0} *\cir<2pt>{} ="F",
	(43,-5) *++!D{0} *\cir<2pt>{} ="F'",
	\ar@{-} "AAAA";"AAA"
	\ar@{-} "AAA";"AA"
	\ar@{-} "AA";"A"
	\ar@{-} "A";"A_1"
	\ar@{.} "A_1";"A_2"^*!U{\cdots}
	\ar@{-} "A_2";"B"
	\ar@{-} "B";"C"
	\ar@{-} "C";"DD"
	\ar@{-} "DD";"D"
	\ar@{-} "D";"D_1"
	\ar@{.} "D_1";"D_2"^*!U{\cdots}
	\ar@{-} "D_2";"E"
	\ar@{-} "E";"F"
	\ar@{-} "E";"F'"
\end{xy} \\ 
$[(2m+1)^2]$ & \begin{xy}
	(0,0) *++!D{0} *\cir<2pt>{}        ="AAAA",
	(3,0) *++!D{2} *\cir<2pt>{}        ="AAA",
	(6,0)*++!D{0} *\cir<2pt>{}        ="AA",
	(9,0)*++!D{2} *\cir<2pt>{}        ="A",
	(12,0) = "A_1",
	(17,0) = "A_2",
	(20,0) *++!D{0} *\cir<2pt>{} ="B",
	(23,0) *++!D{2} *\cir<2pt>{} ="C",
	(26,0) *++!D{0} *\cir<2pt>{} ="DD",
	(29,0) *++!D{2} *\cir<2pt>{} ="D",
	(32,0) = "D_1",
	(37,0) = "D_2",
	(40,0) *++!D{0} *\cir<2pt>{} ="E",
	(43,5) *++!D{2} *\cir<2pt>{} ="F",
	(43,-5) *++!D{2} *\cir<2pt>{} ="F'",
	\ar@{-} "AAAA";"AAA"
	\ar@{-} "AAA";"AA"
	\ar@{-} "AA";"A"
	\ar@{-} "A";"A_1"
	\ar@{.} "A_1";"A_2"^*!U{\cdots}
	\ar@{-} "A_2";"B"
	\ar@{-} "B";"C"
	\ar@{-} "C";"DD"
	\ar@{-} "DD";"D"
	\ar@{-} "D";"D_1"
	\ar@{.} "D_1";"D_2"^*!U{\cdots}
	\ar@{-} "D_2";"E"
	\ar@{-} "E";"F"
	\ar@{-} "E";"F'"
\end{xy} \\ 
\bottomrule
\end{longtable}
\end{center}

Let $\mathfrak{g}$ be a non-compact real form of $\mathfrak{g}_\mathbb{C}$.
Then the Satake diagram $S_\mathfrak{g}$ and $\Psi(\mathfrak{b})$ are given as follows:
\begin{center}
	\begin{longtable}{lll} \toprule
	$\mathfrak{g}$ & $S_\mathfrak{g}$ & $\Psi(\mathfrak{b})$ \\ \midrule

\raisebox{-0.5cm}{\shortstack{$\mathfrak{so}(p,q)$ \\ $(p+q=4m+2)$ \\ $(p>q+2)$}}& 
\begin{xy}
	*\cir<2pt>{}        ="A",
	(5,0) *\cir<2pt>{} ="B",
	(15,0) *+!U{\alpha_q} *\cir<2pt>{} ="C",
	(20,0) *{\bullet} ="D",
	(30,0) *{\bullet} ="E",	
	(35,-5) *++!U{} *{\bullet} ="F",
	(35,5) *++!D{} *{\bullet} ="G",
	\ar@{-} "A";"B"
	\ar@{-} "B"; (8,0)
	\ar@{.} (8,0) ; (12,0)
	\ar@{-} (12,0) ; "C"
	\ar@{-} "C";"D"
	\ar@{-} "D"; (23,0)
	\ar@{.} (23,0) ; (27,0)
	\ar@{-} (27,0) ; "E"
	\ar@{-} "E";"F"
	\ar@{-} "E";"G"	
\end{xy} 
&
$\left\{
\begin{xy}
	*++!D{b_1} *\cir<2pt>{}        ="A",
	(5,0) *++!D{b_2} *\cir<2pt>{} ="B",
	(15,0) *++!D{b_q}  *\cir<2pt>{} ="C",
	(20,0) *++!D{0} *\cir<2pt>{} ="D",
	(30,0) *++!D{0} *\cir<2pt>{} ="E",	
	(35,-5) *++!U{0} *\cir<2pt>{} ="F",
	(35,5) *++!D{0} *\cir<2pt>{} ="G",
	\ar@{-} "A";"B"
	\ar@{-} "B"; (8,0)
	\ar@{.} (8,0) ; (12,0)
	\ar@{-} (12,0) ; "C"
	\ar@{-} "C";"D"
	\ar@{-} "D"; (23,0)
	\ar@{.} (23,0) ; (27,0)
	\ar@{-} (27,0) ; "E"
	\ar@{-} "E";"F"
	\ar@{-} "E";"G"	
\end{xy}
\right\}$
\\ \hline

$\mathfrak{so}(2m+2,2m)$ & \begin{xy}
	*\cir<2pt>{}       ="A",
	(5,0) *\cir<2pt>{} ="B",
	(10,0) *\cir<2pt>{} ="C",
	(20,0) *\cir<2pt>{} ="D",
	(25,0) *\cir<2pt>{} ="E",
	(30,0) *\cir<2pt>{} ="F",	
	(35,-5) *++!U{} *\cir<2pt>{} ="G",
	(35,5) *++!D{} *\cir<2pt>{} ="H",	
	\ar@{-} "A";"B"
	\ar@{-} "B";"C"	
	\ar@{-} "C"; (13,0)
	\ar@{.} (13,0) ; (17,0)
	\ar@{-} (17,0) ; "D"
	\ar@{-} "D";"E"
	\ar@{-} "E";"F"
	\ar@{-} "F";"G"
	\ar@{-} "F";"H"
	\ar@/_2mm/@{<->}@<-1mm> "G";"H"
\end{xy} &
$\left\{
\begin{xy}
	*++!D{b_1} *\cir<2pt>{}       ="A",
	(8,0) *++!D{b_2} *\cir<2pt>{} ="B",
	(10,0) ="C",
	(14,0) ="C'",
	(18,0) = "C''",
	(22,0) *++!D{b_{2m-2}} *\cir<2pt>{} ="E",
	(30,0) *++!L{b_{2m-1}} *\cir<2pt>{} ="F",	
	(35,-5) *++!U{b_{2m}} *\cir<2pt>{} ="G",
	(35,5) *++!D{b_{2m}} *\cir<2pt>{} ="H",	
	\ar@{-} "A";"B"
	\ar@{-} "B";"C"	
	\ar@{-} "C";"C'"
	\ar@{.} "C'"; "C''"
	\ar@{-} "C''" ; "E"
	\ar@{-} "E";"F"
	\ar@{-} "F";"G"
	\ar@{-} "F";"H"
\end{xy}
\right\}$ \\ \hline

$\mathfrak{so}(2m+1,2m+1)$ & \begin{xy}
	*\cir<2pt>{}       ="A",
	(5,0) *\cir<2pt>{} ="B",
	(10,0) *\cir<2pt>{} ="C",
	(20,0) *\cir<2pt>{} ="D",
	(25,0) *\cir<2pt>{} ="E",
	(30,0) *\cir<2pt>{} ="F",	
	(35,-5) *++!U{} *\cir<2pt>{} ="G",
	(35,5) *++!D{} *\cir<2pt>{} ="H",	
	\ar@{-} "A";"B"
	\ar@{-} "B";"C"	
	\ar@{-} "C"; (13,0)
	\ar@{.} (13,0) ; (17,0)
	\ar@{-} (17,0) ; "D"
	\ar@{-} "D";"E"
	\ar@{-} "E";"F"
	\ar@{-} "F";"G"
	\ar@{-} "F";"H"
\end{xy} &
$\left\{
\begin{xy}
	*++!D{b_1} *\cir<2pt>{}       ="A",
	(8,0) *++!D{b_2} *\cir<2pt>{} ="B",
	(12,0) ="C",
	(16,0) ="C'",
	(22,0) *++!D{b_{2m-2}} *\cir<2pt>{} ="E",
	(30,0) *++!L{b_{2m-1}} *\cir<2pt>{} ="F",	
	(35,-5) *++!U{b_{2m}} *\cir<2pt>{} ="G",
	(35,5) *++!D{b_{2m}} *\cir<2pt>{} ="H",	
	\ar@{-} "A";"B"
	\ar@{-} "B";"C"	
	\ar@{.} "C";"C'"^*U{\cdots}
	\ar@{-} "C'"; "E"
	\ar@{-} "E";"F"
	\ar@{-} "F";"G"
	\ar@{-} "F";"H"
\end{xy}
\right\}$ \\ \hline

$\mathfrak{so}^*(4m+2)$ & \begin{xy}
	*{\bullet}       ="A",
	(5,0) *\cir<2pt>{} ="B",
	(10,0) *{\bullet} ="C",
	(20,0) *{\bullet} ="D",
	(25,0) *\cir<2pt>{} ="E",
	(30,0)  *{\bullet} ="F",	
	(35,-5) *++!U{} *\cir<2pt>{} ="G",
	(35,5) *++!D{} *\cir<2pt>{} ="H",	
	\ar@{-} "A";"B"
	\ar@{-} "B";"C"	
	\ar@{-} "C"; (13,0)
	\ar@{.} (13,0) ; (17,0)
	\ar@{-} (17,0) ; "D"
	\ar@{-} "D";"E"
	\ar@{-} "E";"F"
	\ar@{-} "F";"G"
	\ar@{-} "F";"H"
	\ar@/_2mm/@{<->}@<-1mm> "G";"H"	
\end{xy} &
$\left\{
\begin{xy}
	*++!D{0} *\cir<2pt>{}       ="A",
	(5,0) *++!D{b_1} *\cir<2pt>{} ="B",
	(10,0) *++!D{0} *\cir<2pt>{} ="C",
	(20,0) *++!D{0} *\cir<2pt>{} ="D",
	(25,0) *++!D{b_{m-1}} *\cir<2pt>{} ="E",
	(30,0)  *++!L{0} *\cir<2pt>{} ="F",	
	(35,-5) *++!U{b_m} *\cir<2pt>{} ="G",
	(35,5) *++!D{b_m} *\cir<2pt>{} ="H",	
	\ar@{-} "A";"B"
	\ar@{-} "B";"C"	
	\ar@{-} "C"; (13,0)
	\ar@{.} (13,0) ; (17,0)^*U{\cdots}
	\ar@{-} (17,0) ; "D"
	\ar@{-} "D";"E"
	\ar@{-} "E";"F"
	\ar@{-} "F";"G"
	\ar@{-} "F";"H"
\end{xy}
\right\}$
\\ 
\bottomrule
	\end{longtable}
\end{center}

Therefore for each $\mathfrak{g}$,
we can find a basis of $\Psi(\mathfrak{b})$ 
by taking some weighted Dynkin diagrams 
in $\Psi(\mathcal{H}^n(\mathfrak{a}_+)) = 
\{ \Psi_A \in \Psi(\mathcal{H}^n(\mathfrak{j}_+)) \mid \text{$\Psi_A$ matches $S_\mathfrak{g}$}\}$ as follows:
\begin{center}
\begin{longtable}{ll} \toprule
$\mathfrak{g}$ & Example of basis \\ \midrule
$\mathfrak{so}(p,q) \ (p+q = 4m+2,p>q+2)$ & $[3,1^{4m-1}],[5,1^{4m-3}],\dots,[2q+1,1^{4m-2q+1}]$\\
$\mathfrak{so}(2m+2,2m)$ & $[3,1^{4m-1}],[5,1^{4m-3}],\dots,[4m+1,1]$\\
$\mathfrak{so}(2m+1,2m+1)$ & $[3,1^{4m-1}],[5,1^{4m-3}],\dots,[4m+1,1]$\\
$\mathfrak{so}^*(4m+2)$ & $[3^2,1^{4m-4}],[5^2,1^{4m-8}],\dots,[(2m+1)^2]$\\
\bottomrule
\end{longtable}
\end{center}

\subsection{Type $E_6$}

Let us consider the case where 
$\mathfrak{g}_\mathbb{C}$ is of type $E_6$,
that is, 
$\mathfrak{g}_\mathbb{C} \simeq \mathfrak{e}_{6,\mathbb{C}}$.
Then we have
\[
\Map(\Pi,\mathbb{R})^\iota 
= \left\{ 
\begin{xy}
	*++!D{a_1} *\cir<2pt>{}        ="A",
	(10,0) *++!D{a_2} *\cir<2pt>{} ="B",
	(20,0) *++!D{a_3} *\cir<2pt>{} ="C",
	(30,0) *++!D{a_4} *\cir<2pt>{} ="D",
	(40,0) *++!D{a_5} *\cir<2pt>{} ="E",
	(20,-10) *++!U{} *++!L{a_6} *\cir<2pt>{} ="F",
	\ar@{-} "A";"B"
	\ar@{-} "B";"C"
	\ar@{-} "C";"D"
	\ar@{-} "D";"E"
	\ar@{-} "C";"F"
\end{xy} ~\middle|~
a_1 = a_5, \ a_2 = a_4
\right\}
\]

In \cite[Section 8.4]{Collingwood-McGovern93}, 
we can find some examples of weighted Dynkin diagrams in 
$\Psi(\mathcal{H}^n(\mathfrak{j}_+))$ as follows:
\begin{center}
\begin{longtable}{ll} \toprule
	Symbol & Weighted Dynkin diagram in $\Psi(\mathcal{H}^n(\mathfrak{j}_+))$ \\ \midrule
	
	$A_2$ & \begin{xy}
	*++!D{0} *\cir<2pt>{}        ="A",
	(10,0) *++!D{0} *\cir<2pt>{} ="B",
	(20,0) *++!D{0} *\cir<2pt>{} ="C",
	(30,0) *++!D{0} *\cir<2pt>{} ="D",
	(40,0) *++!D{0} *\cir<2pt>{} ="E",
	(20,-10) *++!U{} *++!L{2} *\cir<2pt>{} ="F",
	\ar@{-} "A";"B"
	\ar@{-} "B";"C"
	\ar@{-} "C";"D"
	\ar@{-} "D";"E"
	\ar@{-} "C";"F"
\end{xy}
\\ 
	$2A_2$ & \begin{xy}
	*++!D{2} *\cir<2pt>{}        ="A",
	(10,0) *++!D{0} *\cir<2pt>{} ="B",
	(20,0) *++!D{0} *\cir<2pt>{} ="C",
	(30,0) *++!D{0} *\cir<2pt>{} ="D",
	(40,0) *++!D{2} *\cir<2pt>{} ="E",
	(20,-10) *++!U{} *++!L{0} *\cir<2pt>{} ="F",
	\ar@{-} "A";"B"
	\ar@{-} "B";"C"
	\ar@{-} "C";"D"
	\ar@{-} "D";"E"
	\ar@{-} "C";"F"
\end{xy}
\\ 
	$D_4$ & \begin{xy}
	*++!D{0} *\cir<2pt>{}        ="A",
	(10,0) *++!D{0} *\cir<2pt>{} ="B",
	(20,0) *++!D{2} *\cir<2pt>{} ="C",
	(30,0) *++!D{0} *\cir<2pt>{} ="D",
	(40,0) *++!D{0} *\cir<2pt>{} ="E",
	(20,-10) *++!U{} *++!L{2} *\cir<2pt>{} ="F",
	\ar@{-} "A";"B"
	\ar@{-} "B";"C"
	\ar@{-} "C";"D"
	\ar@{-} "D";"E"
	\ar@{-} "C";"F"
\end{xy}
\\ 
	$E_6$ & \begin{xy}
	*++!D{2} *\cir<2pt>{}        ="A",
	(10,0) *++!D{2} *\cir<2pt>{} ="B",
	(20,0) *++!D{2} *\cir<2pt>{} ="C",
	(30,0) *++!D{2} *\cir<2pt>{} ="D",
	(40,0) *++!D{2} *\cir<2pt>{} ="E",
	(20,-10) *++!U{} *++!L{2} *\cir<2pt>{} ="F",
	\ar@{-} "A";"B"
	\ar@{-} "B";"C"
	\ar@{-} "C";"D"
	\ar@{-} "D";"E"
	\ar@{-} "C";"F"
\end{xy}
\\ 
\bottomrule
\end{longtable}
\end{center}

Let $\mathfrak{g}$ be a non-compact real form of $\mathfrak{g}_\mathbb{C}$.
Then the Satake diagram $S_\mathfrak{g}$ and $\Psi(\mathfrak{b})$ are given as follows:
\begin{center}
	\begin{longtable}{lll} \toprule
	$\mathfrak{g}$ & $S_\mathfrak{g}$ & $\Psi(\mathfrak{b})$ \\ \midrule
	
$\mathfrak{e}_{6(6)}$ & \begin{xy}
	*\cir<2pt>{}        ="A",
	(10,0) *\cir<2pt>{} ="B",
	(20,0) *\cir<2pt>{} ="C",
	(30,0) *\cir<2pt>{} ="D",
	(40,0) *\cir<2pt>{} ="E",
	(20,-10) *++!U{} *\cir<2pt>{} ="F",
	\ar@{-} "A";"B"
	\ar@{-} "B";"C"
	\ar@{-} "C";"D"
	\ar@{-} "D";"E"
	\ar@{-} "C";"F"
\end{xy}&
$\left\{ 
\begin{xy}
	*++!D{b_1} *\cir<2pt>{}        ="A",
	(10,0) *++!D{b_2} *\cir<2pt>{} ="B",
	(20,0) *++!D{b_3} *\cir<2pt>{} ="C",
	(30,0) *++!D{b_2} *\cir<2pt>{} ="D",
	(40,0) *++!D{b_1} *\cir<2pt>{} ="E",
	(20,-10) *++!U{} *++!L{b_4} *\cir<2pt>{} ="F",
	\ar@{-} "A";"B"
	\ar@{-} "B";"C"
	\ar@{-} "C";"D"
	\ar@{-} "D";"E"
	\ar@{-} "C";"F"
\end{xy} 
\right\}$
\\ \hline
	
$\mathfrak{e}_{6(2)}$ & \begin{xy}
	*\cir<2pt>{}        ="A",
	(10,0) *\cir<2pt>{} ="B",
	(20,0) *\cir<2pt>{} ="C",
	(30,0) *\cir<2pt>{} ="D",
	(40,0) *\cir<2pt>{} ="E",
	(20,-10) *++!U{} *\cir<2pt>{} ="F",
	\ar@{-} "A";"B"
	\ar@{-} "B";"C"
	\ar@{-} "C";"D"
	\ar@{-} "D";"E"
	\ar@{-} "C";"F"
	
	\ar@(ru,lu) @<1mm> @{<->} "A" ; "E"
	\ar@/^4mm/ @<1mm> @{<->} "B" ; "D"	
\end{xy}
&
$\left\{ 
\begin{xy}
	*++!D{b_1} *\cir<2pt>{}        ="A",
	(10,0) *++!D{b_2} *\cir<2pt>{} ="B",
	(20,0) *++!D{b_3} *\cir<2pt>{} ="C",
	(30,0) *++!D{b_2} *\cir<2pt>{} ="D",
	(40,0) *++!D{b_1} *\cir<2pt>{} ="E",
	(20,-10) *++!U{} *++!L{b_4} *\cir<2pt>{} ="F",
	\ar@{-} "A";"B"
	\ar@{-} "B";"C"
	\ar@{-} "C";"D"
	\ar@{-} "D";"E"
	\ar@{-} "C";"F"
\end{xy} 
\right\}$
\\ \hline
	
$\mathfrak{e}_{6(-14)}$ & \begin{xy}
	*\cir<2pt>{}        ="A",
	(10,0) *{\bullet} ="B",
	(20,0) *{\bullet} ="C",
	(30,0) *{\bullet} ="D",
	(40,0) *\cir<2pt>{} ="E",
	(20,-10) *++!U{} *\cir<2pt>{} ="F",
	\ar@{-} "A";"B"
	\ar@{-} "B";"C"
	\ar@{-} "C";"D"
	\ar@{-} "D";"E"
	\ar@{-} "C";"F"
	
	\ar@(ru,lu) @<1mm> @{<->} "A" ; "E"
\end{xy}
&
$\left\{ 
\begin{xy}
	*++!D{b_1} *\cir<2pt>{}        ="A",
	(10,0) *++!D{0} *\cir<2pt>{} ="B",
	(20,0) *++!D{0} *\cir<2pt>{} ="C",
	(30,0) *++!D{0} *\cir<2pt>{} ="D",
	(40,0) *++!D{b_1} *\cir<2pt>{} ="E",
	(20,-10) *++!U{} *++!L{b_2} *\cir<2pt>{} ="F",
	\ar@{-} "A";"B"
	\ar@{-} "B";"C"
	\ar@{-} "C";"D"
	\ar@{-} "D";"E"
	\ar@{-} "C";"F"
\end{xy} 
\right\}$
\\ \hline	

$\mathfrak{e}_{6(-26)}$ & \begin{xy}
	*\cir<2pt>{}        ="A",
	(10,0) *{\bullet} ="B",
	(20,0) *{\bullet} ="C",
	(30,0) *{\bullet} ="D",
	(40,0) *\cir<2pt>{} ="E",
	(20,-10) *++!U{} *{\bullet} ="F",
	\ar@{-} "A";"B"
	\ar@{-} "B";"C"
	\ar@{-} "C";"D"
	\ar@{-} "D";"E"
	\ar@{-} "C";"F"
\end{xy}
&
$\left\{ 
\begin{xy}
	*++!D{b} *\cir<2pt>{}        ="A",
	(10,0) *++!D{0} *\cir<2pt>{} ="B",
	(20,0) *++!D{0} *\cir<2pt>{} ="C",
	(30,0) *++!D{0} *\cir<2pt>{} ="D",
	(40,0) *++!D{b} *\cir<2pt>{} ="E",
	(20,-10) *++!U{} *++!L{0} *\cir<2pt>{} ="F",
	\ar@{-} "A";"B"
	\ar@{-} "B";"C"
	\ar@{-} "C";"D"
	\ar@{-} "D";"E"
	\ar@{-} "C";"F"
\end{xy} 
\right\}$
\\ \bottomrule
	\end{longtable}
\end{center}

Therefore for each $\mathfrak{g}$,
we can find a basis of $\Psi(\mathfrak{b})$ 
by taking some weighted Dynkin diagrams 
in $\Psi(\mathcal{H}^n(\mathfrak{a}_+)) = 
\{ \Psi_A \in \Psi(\mathcal{H}^n(\mathfrak{j}_+)) \mid \text{$\Psi_A$ matches $S_\mathfrak{g}$}\}$ as follows:

\begin{center}
\begin{longtable}{ll} \toprule
$\mathfrak{g}$ & Example of basis of $\Psi(\mathfrak{b})$ \\ \midrule
$\mathfrak{e}_{6(6)}$ & $A_2, 2A_2, D_4, E_6$ \\ \hline
$\mathfrak{e}_{6(2)}$ & $A_2, 2A_2, D_4, E_6$ \\ \hline
$\mathfrak{e}_{6(-14)}$ & $A_2,2A_2$ \\ \hline
$\mathfrak{e}_{6(-26)}$ & $2A_2$ \\
\bottomrule
\end{longtable}
\end{center}

\subsection{Type $E_7$}

Let us consider the case where 
$\mathfrak{g}_\mathbb{C}$ is of type $E_7$,
that is, 
$\mathfrak{g}_\mathbb{C} \simeq \mathfrak{e}_{7,\mathbb{C}}$.
Then we have
\[
\Map(\Pi,\mathbb{R})^\iota 
= \Map(\Pi,\mathbb{R}) 
= \left\{ 
\begin{xy}
	*++!D{a_1} *\cir<2pt>{}        ="A",
	(10,0) *++!D{a_2} *\cir<2pt>{} ="B",
	(20,0) *++!D{a_3} *\cir<2pt>{} ="C",
	(30,0) *++!D{a_4} *\cir<2pt>{} ="D",
	(40,0) *++!D{a_5} *\cir<2pt>{} ="E",
	(50,0) *++!D{a_6} *\cir<2pt>{} ="F",
	(30,-10) *++!U{} *++!L{a_7} *\cir<2pt>{} ="G",
	\ar@{-} "A";"B"
	\ar@{-} "B";"C"
	\ar@{-} "C";"D"
	\ar@{-} "D";"E"
	\ar@{-} "E";"F"
	\ar@{-} "D";"G"
\end{xy}
\right\}
\]

In \cite[Section 8.4]{Collingwood-McGovern93}, 
we can find some examples of weighted Dynkin diagrams in 
$\Psi(\mathcal{H}^n(\mathfrak{j}_+))$ as follows:
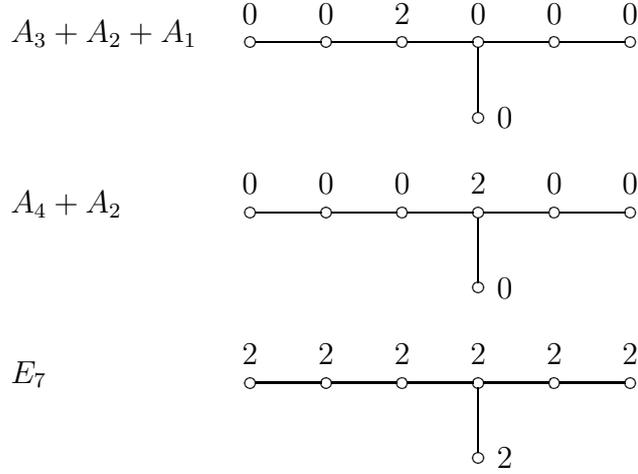
\begin{center}
\begin{longtable}{ll} \toprule
	Symbol & Weighted Dynkin diagram in $\Psi(\mathcal{H}^n(\mathfrak{j}_+))$ \\ \midrule
	
	$(3A_1)''$ & \begin{xy}
	*++!D{2} *\cir<2pt>{}        ="A",
	(10,0) *++!D{0} *\cir<2pt>{} ="B",
	(20,0) *++!D{0} *\cir<2pt>{} ="C",
	(30,0) *++!D{0} *\cir<2pt>{} ="D",
	(40,0) *++!D{0} *\cir<2pt>{} ="E",
	(50,0) *++!D{0} *\cir<2pt>{} ="F",
	(30,-10) *++!U{} *++!L{0} *\cir<2pt>{} ="G",
	\ar@{-} "A";"B"
	\ar@{-} "B";"C"
	\ar@{-} "C";"D"
	\ar@{-} "D";"E"
	\ar@{-} "E";"F"
	\ar@{-} "D";"G"
\end{xy}
\\ 
	$A_2$ & \begin{xy}
	*++!D{0} *\cir<2pt>{}        ="A",
	(10,0) *++!D{0} *\cir<2pt>{} ="B",
	(20,0) *++!D{0} *\cir<2pt>{} ="C",
	(30,0) *++!D{0} *\cir<2pt>{} ="D",
	(40,0) *++!D{0} *\cir<2pt>{} ="E",
	(50,0) *++!D{2} *\cir<2pt>{} ="F",
	(30,-10) *++!U{} *++!L{0} *\cir<2pt>{} ="G",
	\ar@{-} "A";"B"
	\ar@{-} "B";"C"
	\ar@{-} "C";"D"
	\ar@{-} "D";"E"
	\ar@{-} "E";"F"
	\ar@{-} "D";"G"
\end{xy}
\\ 
	$2A_2$ & \begin{xy}
	*++!D{0} *\cir<2pt>{}        ="A",
	(10,0) *++!D{2} *\cir<2pt>{} ="B",
	(20,0) *++!D{0} *\cir<2pt>{} ="C",
	(30,0) *++!D{0} *\cir<2pt>{} ="D",
	(40,0) *++!D{0} *\cir<2pt>{} ="E",
	(50,0) *++!D{0} *\cir<2pt>{} ="F",
	(30,-10) *++!U{} *++!L{0} *\cir<2pt>{} ="G",
	\ar@{-} "A";"B"
	\ar@{-} "B";"C"
	\ar@{-} "C";"D"
	\ar@{-} "D";"E"
	\ar@{-} "E";"F"
	\ar@{-} "D";"G"
\end{xy}
\\ 
	$D_4$ & \begin{xy}
	*++!D{0} *\cir<2pt>{}        ="A",
	(10,0) *++!D{0} *\cir<2pt>{} ="B",
	(20,0) *++!D{0} *\cir<2pt>{} ="C",
	(30,0) *++!D{0} *\cir<2pt>{} ="D",
	(40,0) *++!D{2} *\cir<2pt>{} ="E",
	(50,0) *++!D{2} *\cir<2pt>{} ="F",
	(30,-10) *++!U{} *++!L{0} *\cir<2pt>{} ="G",
	\ar@{-} "A";"B"
	\ar@{-} "B";"C"
	\ar@{-} "C";"D"
	\ar@{-} "D";"E"
	\ar@{-} "E";"F"
	\ar@{-} "D";"G"
\end{xy}
\\ 
	$A_3+A_2+A_1$ & \begin{xy}
	*++!D{0} *\cir<2pt>{}        ="A",
	(10,0) *++!D{0} *\cir<2pt>{} ="B",
	(20,0) *++!D{2} *\cir<2pt>{} ="C",
	(30,0) *++!D{0} *\cir<2pt>{} ="D",
	(40,0) *++!D{0} *\cir<2pt>{} ="E",
	(50,0) *++!D{0} *\cir<2pt>{} ="F",
	(30,-10) *++!U{} *++!L{0} *\cir<2pt>{} ="G",
	\ar@{-} "A";"B"
	\ar@{-} "B";"C"
	\ar@{-} "C";"D"
	\ar@{-} "D";"E"
	\ar@{-} "E";"F"
	\ar@{-} "D";"G"
\end{xy}
\\ 
	$A_4+A_2$ & \begin{xy}
	*++!D{0} *\cir<2pt>{}        ="A",
	(10,0) *++!D{0} *\cir<2pt>{} ="B",
	(20,0) *++!D{0} *\cir<2pt>{} ="C",
	(30,0) *++!D{2} *\cir<2pt>{} ="D",
	(40,0) *++!D{0} *\cir<2pt>{} ="E",
	(50,0) *++!D{0} *\cir<2pt>{} ="F",
	(30,-10) *++!U{} *++!L{0} *\cir<2pt>{} ="G",
	\ar@{-} "A";"B"
	\ar@{-} "B";"C"
	\ar@{-} "C";"D"
	\ar@{-} "D";"E"
	\ar@{-} "E";"F"
	\ar@{-} "D";"G"
\end{xy}
\\ 
	$E_7$ & \begin{xy}
	*++!D{2} *\cir<2pt>{}        ="A",
	(10,0) *++!D{2} *\cir<2pt>{} ="B",
	(20,0) *++!D{2} *\cir<2pt>{} ="C",
	(30,0) *++!D{2} *\cir<2pt>{} ="D",
	(40,0) *++!D{2} *\cir<2pt>{} ="E",
	(50,0) *++!D{2} *\cir<2pt>{} ="F",
	(30,-10) *++!U{} *++!L{2} *\cir<2pt>{} ="G",
	\ar@{-} "A";"B"
	\ar@{-} "B";"C"
	\ar@{-} "C";"D"
	\ar@{-} "D";"E"
	\ar@{-} "E";"F"
	\ar@{-} "D";"G"
\end{xy}
\\ 
\bottomrule
\end{longtable}
\end{center}

Let $\mathfrak{g}$ be a non-compact real form of $\mathfrak{g}_\mathbb{C}$.
Then the Satake diagram $S_\mathfrak{g}$ and $\Psi(\mathfrak{b})$ are given as follows:
\begin{center}
	\begin{longtable}{lll} \toprule
	$\mathfrak{g}$ & $S_\mathfrak{g}$ & $\Psi(\mathfrak{b})$ \\ \midrule
	
$\mathfrak{e}_{7(7)}$ & \begin{xy}
	*\cir<2pt>{}        ="A",
	(8,0) *\cir<2pt>{} ="B",
	(16,0) *\cir<2pt>{} ="C",
	(24,0) *\cir<2pt>{} ="D",
	(32,0) *\cir<2pt>{} ="E",
	(40,0) *\cir<2pt>{} ="F",
	(24,-8) *++!U{} *\cir<2pt>{} ="G",
	\ar@{-} "A";"B"
	\ar@{-} "B";"C"
	\ar@{-} "C";"D"
	\ar@{-} "D";"E"
	\ar@{-} "E";"F"
	\ar@{-} "D";"G"
\end{xy}
& $\left\{ 
\begin{xy}
	*++!D{b_1} *\cir<2pt>{}        ="A",
	(8,0) *++!D{b_2} *\cir<2pt>{} ="B",
	(16,0) *++!D{b_3} *\cir<2pt>{} ="C",
	(24,0) *++!D{b_4} *\cir<2pt>{} ="D",
	(32,0) *++!D{b_5} *\cir<2pt>{} ="E",
	(40,0) *++!D{b_6} *\cir<2pt>{} ="F",
	(24,-8) *++!U{} *++!L{b_7} *\cir<2pt>{} ="G",
	\ar@{-} "A";"B"
	\ar@{-} "B";"C"
	\ar@{-} "C";"D"
	\ar@{-} "D";"E"
	\ar@{-} "E";"F"
	\ar@{-} "D";"G"
\end{xy}
\right\}$
 \\ \hline
	
$\mathfrak{e}_{7(-5)}$ & \begin{xy}
	*{\bullet}        ="A",
	(8,0) *\cir<2pt>{} ="B",
	(16,0) *{\bullet} ="C",
	(24,0) *\cir<2pt>{} ="D",
	(32,0) *\cir<2pt>{} ="E",
	(40,0) *\cir<2pt>{} ="F",
	(24,-8) *++!U{} *{\bullet} ="G",
	\ar@{-} "A";"B"
	\ar@{-} "B";"C"
	\ar@{-} "C";"D"
	\ar@{-} "D";"E"
	\ar@{-} "E";"F"
	\ar@{-} "D";"G"
\end{xy} & $\left\{ 
\begin{xy}
	*++!D{0} *\cir<2pt>{}        ="A",
	(8,0) *++!D{b_1} *\cir<2pt>{} ="B",
	(16,0) *++!D{0} *\cir<2pt>{} ="C",
	(24,0) *++!D{b_2} *\cir<2pt>{} ="D",
	(32,0) *++!D{b_3} *\cir<2pt>{} ="E",
	(40,0) *++!D{b_4} *\cir<2pt>{} ="F",
	(24,-8) *++!U{} *++!L{0} *\cir<2pt>{} ="G",
	\ar@{-} "A";"B"
	\ar@{-} "B";"C"
	\ar@{-} "C";"D"
	\ar@{-} "D";"E"
	\ar@{-} "E";"F"
	\ar@{-} "D";"G"
\end{xy}
\right\}$
\\ \hline
	
$\mathfrak{e}_{7(-25)}$ & \begin{xy}
	*\cir<2pt>{}        ="A",
	(8,0) *\cir<2pt>{} ="B",
	(16,0) *{\bullet} ="C",
	(24,0) *{\bullet} ="D",
	(32,0) *{\bullet} ="E",
	(40,0) *\cir<2pt>{} ="F",
	(24,-8) *++!U{} *{\bullet} ="G",
	\ar@{-} "A";"B"
	\ar@{-} "B";"C"
	\ar@{-} "C";"D"
	\ar@{-} "D";"E"
	\ar@{-} "E";"F"
	\ar@{-} "D";"G"
\end{xy} & $\left\{ 
\begin{xy}
	*++!D{b_1} *\cir<2pt>{}        ="A",
	(8,0) *++!D{b_2} *\cir<2pt>{} ="B",
	(16,0) *++!D{0} *\cir<2pt>{} ="C",
	(24,0) *++!D{0} *\cir<2pt>{} ="D",
	(32,0) *++!D{0} *\cir<2pt>{} ="E",
	(40,0) *++!D{b_3} *\cir<2pt>{} ="F",
	(24,-8) *++!U{} *++!L{0} *\cir<2pt>{} ="G",
	\ar@{-} "A";"B"
	\ar@{-} "B";"C"
	\ar@{-} "C";"D"
	\ar@{-} "D";"E"
	\ar@{-} "E";"F"
	\ar@{-} "D";"G"
\end{xy}
\right\}$
\\
\bottomrule
	\end{longtable}
\end{center}

Therefore for each $\mathfrak{g}$,
we can find a basis of $\Psi(\mathfrak{b})$ 
by taking some weighted Dynkin diagrams 
in $\Psi(\mathcal{H}^n(\mathfrak{a}_+)) = 
\{ \Psi_A \in \Psi(\mathcal{H}^n(\mathfrak{j}_+)) \mid \text{$\Psi_A$ matches $S_\mathfrak{g}$}\}$ as follows:

\begin{center}
\begin{longtable}{ll} \toprule
$\mathfrak{g}$ & Example of basis of $\Psi(\mathfrak{b})$ \\ \midrule
$\mathfrak{e}_{7(7)}$ & $3A_1'',A_2,2A_2,D_4,A_3+A_2+A_1,A_4+A_2,E_7$\\ \hline
$\mathfrak{e}_{7(-5)}$ & $A_2,2A_2,D_4,A_4+A_2$ \\ \hline
$\mathfrak{e}_{7(-25)}$ & $3A_1'',A_2,2A_2$ \\
\bottomrule
\end{longtable}
\end{center}

\subsection{Type $E_8$}

Let us consider the case where 
$\mathfrak{g}_\mathbb{C}$ is of type $E_8$,
that is, 
$\mathfrak{g}_\mathbb{C} \simeq \mathfrak{e}_{8,\mathbb{C}}$.
Then we have
\[
\Map(\Pi,\mathbb{R})^\iota 
= \Map(\Pi,\mathbb{R})
= \left\{ 
\begin{xy}
	*++!D{a_1} *\cir<2pt>{}        ="A",
	(10,0) *++!D{a_2} *\cir<2pt>{} ="B",
	(20,0) *++!D{a_3} *\cir<2pt>{} ="C",
	(30,0) *++!D{a_4} *\cir<2pt>{} ="D",
	(40,0) *++!D{a_5} *\cir<2pt>{} ="E",
	(50,0) *++!D{a_6} *\cir<2pt>{} ="F",
	(60,0) *++!D{a_7} *\cir<2pt>{} ="G",
	(40,-10) *++!U{} *++!L{a_8} *\cir<2pt>{} ="H",
	\ar@{-} "A";"B"
	\ar@{-} "B";"C"
	\ar@{-} "C";"D"
	\ar@{-} "D";"E"
	\ar@{-} "E";"F"
	\ar@{-} "F";"G"
	\ar@{-} "E";"H"
\end{xy}
\right\}
\]

In \cite[Section 8.4]{Collingwood-McGovern93}, 
we can find some examples of weighted Dynkin diagrams in 
$\Psi(\mathcal{H}^n(\mathfrak{j}_+))$ as follows:
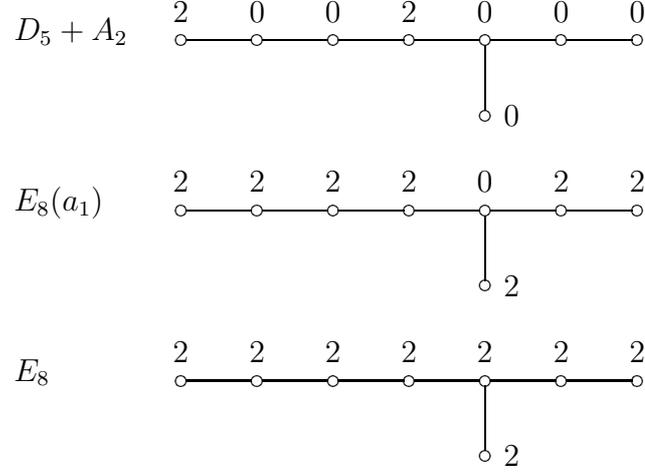
\begin{center}
\begin{longtable}{ll} \toprule
	Symbol & Weighted Dynkin diagram in $\Psi(\mathcal{H}^n(\mathfrak{j}_+))$ \\ \midrule
	$A_2$ & \begin{xy}
	*++!D{2} *\cir<2pt>{}        ="A",
	(10,0) *++!D{0} *\cir<2pt>{} ="B",
	(20,0) *++!D{0} *\cir<2pt>{} ="C",
	(30,0) *++!D{0} *\cir<2pt>{} ="D",
	(40,0) *++!D{0} *\cir<2pt>{} ="E",
	(50,0) *++!D{0} *\cir<2pt>{} ="F",
	(60,0) *++!D{0} *\cir<2pt>{} ="G",
	(40,-10) *++!U{} *++!L{0} *\cir<2pt>{} ="H",
	\ar@{-} "A";"B"
	\ar@{-} "B";"C"
	\ar@{-} "C";"D"
	\ar@{-} "D";"E"
	\ar@{-} "E";"F"
	\ar@{-} "F";"G"
	\ar@{-} "E";"H"
\end{xy}
\\ 
	$2A_2$ & \begin{xy}
	*++!D{0} *\cir<2pt>{}        ="A",
	(10,0) *++!D{0} *\cir<2pt>{} ="B",
	(20,0) *++!D{0} *\cir<2pt>{} ="C",
	(30,0) *++!D{0} *\cir<2pt>{} ="D",
	(40,0) *++!D{0} *\cir<2pt>{} ="E",
	(50,0) *++!D{0} *\cir<2pt>{} ="F",
	(60,0) *++!D{2} *\cir<2pt>{} ="G",
	(40,-10) *++!U{} *++!L{0} *\cir<2pt>{} ="H",
	\ar@{-} "A";"B"
	\ar@{-} "B";"C"
	\ar@{-} "C";"D"
	\ar@{-} "D";"E"
	\ar@{-} "E";"F"
	\ar@{-} "F";"G"
	\ar@{-} "E";"H"
\end{xy}
\\ 
	$D_4$ & \begin{xy}
	*++!D{2} *\cir<2pt>{}        ="A",
	(10,0) *++!D{2} *\cir<2pt>{} ="B",
	(20,0) *++!D{0} *\cir<2pt>{} ="C",
	(30,0) *++!D{0} *\cir<2pt>{} ="D",
	(40,0) *++!D{0} *\cir<2pt>{} ="E",
	(50,0) *++!D{0} *\cir<2pt>{} ="F",
	(60,0) *++!D{0} *\cir<2pt>{} ="G",
	(40,-10) *++!U{} *++!L{0} *\cir<2pt>{} ="H",
	\ar@{-} "A";"B"
	\ar@{-} "B";"C"
	\ar@{-} "C";"D"
	\ar@{-} "D";"E"
	\ar@{-} "E";"F"
	\ar@{-} "F";"G"
	\ar@{-} "E";"H"
\end{xy}
\\ 
	$A_4 + A_2$ & \begin{xy}
	*++!D{0} *\cir<2pt>{}        ="A",
	(10,0) *++!D{0} *\cir<2pt>{} ="B",
	(20,0) *++!D{2} *\cir<2pt>{} ="C",
	(30,0) *++!D{0} *\cir<2pt>{} ="D",
	(40,0) *++!D{0} *\cir<2pt>{} ="E",
	(50,0) *++!D{0} *\cir<2pt>{} ="F",
	(60,0) *++!D{0} *\cir<2pt>{} ="G",
	(40,-10) *++!U{} *++!L{0} *\cir<2pt>{} ="H",
	\ar@{-} "A";"B"
	\ar@{-} "B";"C"
	\ar@{-} "C";"D"
	\ar@{-} "D";"E"
	\ar@{-} "E";"F"
	\ar@{-} "F";"G"
	\ar@{-} "E";"H"
\end{xy}
\\ 
	$D_4+A_2$ & \begin{xy}
	*++!D{2} *\cir<2pt>{}        ="A",
	(10,0) *++!D{0} *\cir<2pt>{} ="B",
	(20,0) *++!D{0} *\cir<2pt>{} ="C",
	(30,0) *++!D{0} *\cir<2pt>{} ="D",
	(40,0) *++!D{0} *\cir<2pt>{} ="E",
	(50,0) *++!D{0} *\cir<2pt>{} ="F",
	(60,0) *++!D{0} *\cir<2pt>{} ="G",
	(40,-10) *++!U{2} *++!L{} *\cir<2pt>{} ="H",
	\ar@{-} "A";"B"
	\ar@{-} "B";"C"
	\ar@{-} "C";"D"
	\ar@{-} "D";"E"
	\ar@{-} "E";"F"
	\ar@{-} "F";"G"
	\ar@{-} "E";"H"
\end{xy}
\\ 
	$D_5+A_2$ & \begin{xy}
	*++!D{2} *\cir<2pt>{}        ="A",
	(10,0) *++!D{0} *\cir<2pt>{} ="B",
	(20,0) *++!D{0} *\cir<2pt>{} ="C",
	(30,0) *++!D{2} *\cir<2pt>{} ="D",
	(40,0) *++!D{0} *\cir<2pt>{} ="E",
	(50,0) *++!D{0} *\cir<2pt>{} ="F",
	(60,0) *++!D{0} *\cir<2pt>{} ="G",
	(40,-10) *++!U{} *++!L{0} *\cir<2pt>{} ="H",
	\ar@{-} "A";"B"
	\ar@{-} "B";"C"
	\ar@{-} "C";"D"
	\ar@{-} "D";"E"
	\ar@{-} "E";"F"
	\ar@{-} "F";"G"
	\ar@{-} "E";"H"
\end{xy}
\\ 
	$E_8(a_1)$ & \begin{xy}
	*++!D{2} *\cir<2pt>{}        ="A",
	(10,0) *++!D{2} *\cir<2pt>{} ="B",
	(20,0) *++!D{2} *\cir<2pt>{} ="C",
	(30,0) *++!D{2} *\cir<2pt>{} ="D",
	(40,0) *++!D{0} *\cir<2pt>{} ="E",
	(50,0) *++!D{2} *\cir<2pt>{} ="F",
	(60,0) *++!D{2} *\cir<2pt>{} ="G",
	(40,-10) *++!U{} *++!L{2} *\cir<2pt>{} ="H",
	\ar@{-} "A";"B"
	\ar@{-} "B";"C"
	\ar@{-} "C";"D"
	\ar@{-} "D";"E"
	\ar@{-} "E";"F"
	\ar@{-} "F";"G"
	\ar@{-} "E";"H"
\end{xy}
\\ 
	$E_8$ & \begin{xy}
	*++!D{2} *\cir<2pt>{}        ="A",
	(10,0) *++!D{2} *\cir<2pt>{} ="B",
	(20,0) *++!D{2} *\cir<2pt>{} ="C",
	(30,0) *++!D{2} *\cir<2pt>{} ="D",
	(40,0) *++!D{2} *\cir<2pt>{} ="E",
	(50,0) *++!D{2} *\cir<2pt>{} ="F",
	(60,0) *++!D{2} *\cir<2pt>{} ="G",
	(40,-10) *++!U{} *++!L{2} *\cir<2pt>{} ="H",
	\ar@{-} "A";"B"
	\ar@{-} "B";"C"
	\ar@{-} "C";"D"
	\ar@{-} "D";"E"
	\ar@{-} "E";"F"
	\ar@{-} "F";"G"
	\ar@{-} "E";"H"
\end{xy}
\\ 
\bottomrule
\end{longtable}
\end{center}

Let $\mathfrak{g}$ be a non-compact real form of $\mathfrak{g}_\mathbb{C}$.
Then the Satake diagram $S_\mathfrak{g}$ and $\Psi(\mathfrak{b})$ are given as follows:
\begin{center}
	\begin{longtable}{lll} \toprule
	$\mathfrak{g}$ & $S_\mathfrak{g}$ & $\Psi(\mathfrak{b})$ \\ \midrule
	
$\mathfrak{e}_{8(8)}$ & \begin{xy}
	*\cir<2pt>{}        ="A",
	(6,0) *\cir<2pt>{} ="B",
	(12,0) *\cir<2pt>{} ="C",
	(18,0) *\cir<2pt>{} ="D",
	(24,0) *\cir<2pt>{} ="E",
	(30,0) *\cir<2pt>{} ="F",
	(36,0) *\cir<2pt>{} ="G",
	(24,-6) *++!U{} *\cir<2pt>{} ="H",
	\ar@{-} "A";"B"
	\ar@{-} "B";"C"
	\ar@{-} "C";"D"
	\ar@{-} "D";"E"
	\ar@{-} "E";"F"
	\ar@{-} "F";"G"
	\ar@{-} "E";"H"
\end{xy} 
& $\left\{ 
\begin{xy}
	*++!D{b_1} *\cir<2pt>{}        ="A",
	(6,0) *++!D{b_2} *\cir<2pt>{} ="B",
	(12,0) *++!D{b_3} *\cir<2pt>{} ="C",
	(18,0) *++!D{b_4} *\cir<2pt>{} ="D",
	(24,0) *++!D{b_5} *\cir<2pt>{} ="E",
	(30,0) *++!D{b_6} *\cir<2pt>{} ="F",
	(36,0) *++!D{b_7} *\cir<2pt>{} ="G",
	(24,-6) *++!U{} *++!L{b_8} *\cir<2pt>{} ="H",
	\ar@{-} "A";"B"
	\ar@{-} "B";"C"
	\ar@{-} "C";"D"
	\ar@{-} "D";"E"
	\ar@{-} "E";"F"
	\ar@{-} "F";"G"
	\ar@{-} "E";"H"
\end{xy}
\right\}$
\\ \hline
	
$\mathfrak{e}_{8(-24)}$ & \begin{xy}
	*\cir<2pt>{}        ="A",
	(6,0) *\cir<2pt>{} ="B",
	(12,0) *\cir<2pt>{} ="C",
	(18,0) *{\bullet} ="D",
	(24,0) *{\bullet} ="E",
	(30,0) *{\bullet} ="F",
	(36,0) *\cir<2pt>{} ="G",
	(24,-6) *++!U{} *{\bullet} ="H",
	\ar@{-} "A";"B"
	\ar@{-} "B";"C"
	\ar@{-} "C";"D"
	\ar@{-} "D";"E"
	\ar@{-} "E";"F"
	\ar@{-} "F";"G"
	\ar@{-} "E";"H"
\end{xy}
& $\left\{ 
\begin{xy}
	*++!D{b_1} *\cir<2pt>{}        ="A",
	(6,0) *++!D{b_2} *\cir<2pt>{} ="B",
	(12,0) *++!D{b_3} *\cir<2pt>{} ="C",
	(18,0) *++!D{0} *\cir<2pt>{} ="D",
	(24,0) *++!D{0} *\cir<2pt>{} ="E",
	(30,0) *++!D{0} *\cir<2pt>{} ="F",
	(36,0) *++!D{b_4} *\cir<2pt>{} ="G",
	(24,-6) *++!U{} *++!L{0} *\cir<2pt>{} ="H",
	\ar@{-} "A";"B"
	\ar@{-} "B";"C"
	\ar@{-} "C";"D"
	\ar@{-} "D";"E"
	\ar@{-} "E";"F"
	\ar@{-} "F";"G"
	\ar@{-} "E";"H"
\end{xy}
\right\}$
\\ 
\bottomrule
	\end{longtable}
\end{center}

Therefore for each $\mathfrak{g}$,
we can find a basis of $\Psi(\mathfrak{b})$ 
by taking some weighted Dynkin diagrams 
in $\Psi(\mathcal{H}^n(\mathfrak{a}_+)) = 
\{ \Psi_A \in \Psi(\mathcal{H}^n(\mathfrak{j}_+)) \mid \text{$\Psi_A$ matches $S_\mathfrak{g}$}\}$ as follows:

\begin{center}
\begin{longtable}{ll} \toprule
$\mathfrak{g}$ & Example of basis of $\Psi(\mathfrak{b})$ \\ \midrule
$\mathfrak{e}_{8(8)}$ & $A_2,2A_2,D_4,A_4+A_2,D_4+A_2,D_5+A_2,E_8(a_1),E_8$ \\ \hline
$\mathfrak{e}_{8(-24)}$ & $A_2,2A_2,D_4,A_4+A_2$ \\
\bottomrule
\end{longtable}
\end{center}

\subsection{Type $F_4$}

Let us consider the case where 
$\mathfrak{g}_\mathbb{C}$ is of type $F_4$,
that is, 
$\mathfrak{g}_\mathbb{C} \simeq \mathfrak{f}_{4,\mathbb{C}}$.
Then we have
\[
\Map(\Pi,\mathbb{R})^\iota = \Map(\Pi,\mathbb{R}) 
= \left\{ \begin{xy}
	*++!D{a_1} *\cir<2pt>{}        ="A",
	(10,0) *++!D{a_2} *\cir<2pt>{} ="B",
	(20,0) *++!D{a_3} *\cir<2pt>{} ="C",
	(30,0) *++!D{a_4} *\cir<2pt>{} ="D",
	\ar@{-} "A";"B"
	\ar@{=>} "B";"C"
	\ar@{-} "C";"D"
\end{xy} 
\right\}
\]

In \cite[Section 8.4]{Collingwood-McGovern93}, 
we can find some examples of weighted Dynkin diagrams in 
$\Psi(\mathcal{H}^n(\mathfrak{j}_+))$ as follows:

\begin{center}
\begin{longtable}{ll} \toprule
	Symbol & Weighted Dynkin diagram in $\Psi(\mathcal{H}^n(\mathfrak{j}_+))$ \\ \midrule
	$A_2$ & \begin{xy}
	*++!D{2} *\cir<2pt>{}        ="A",
	(10,0) *++!D{0} *\cir<2pt>{} ="B",
	(20,0) *++!D{0} *\cir<2pt>{} ="C",
	(30,0) *++!D{0} *\cir<2pt>{} ="D",
	\ar@{-} "A";"B"
	\ar@{=>} "B";"C"
	\ar@{-} "C";"D"
\end{xy} \\
	$\tilde{A_2}$ & \begin{xy}
	*++!D{0} *\cir<2pt>{}        ="A",
	(10,0) *++!D{0} *\cir<2pt>{} ="B",
	(20,0) *++!D{0} *\cir<2pt>{} ="C",
	(30,0) *++!D{2} *\cir<2pt>{} ="D",
	\ar@{-} "A";"B"
	\ar@{=>} "B";"C"
	\ar@{-} "C";"D"
\end{xy} \\
	$B_3$ & \begin{xy}
	*++!D{2} *\cir<2pt>{}        ="A",
	(10,0) *++!D{2} *\cir<2pt>{} ="B",
	(20,0) *++!D{0} *\cir<2pt>{} ="C",
	(30,0) *++!D{0} *\cir<2pt>{} ="D",
	\ar@{-} "A";"B"
	\ar@{=>} "B";"C"
	\ar@{-} "C";"D"
\end{xy} \\
	$F_4$ & \begin{xy}
	*++!D{2} *\cir<2pt>{}        ="A",
	(10,0) *++!D{2} *\cir<2pt>{} ="B",
	(20,0) *++!D{2} *\cir<2pt>{} ="C",
	(30,0) *++!D{2} *\cir<2pt>{} ="D",
	\ar@{-} "A";"B"
	\ar@{=>} "B";"C"
	\ar@{-} "C";"D"
\end{xy} \\
\bottomrule
\end{longtable}
\end{center}

Let $\mathfrak{g}$ be a non-compact real form of $\mathfrak{g}_\mathbb{C}$.
Then the Satake diagram $S_\mathfrak{g}$ and $\Psi(\mathfrak{b})$ are given as follows:

\begin{center}
	\begin{longtable}{lll} \toprule
	$\mathfrak{g}$ & $S_\mathfrak{g}$ & $\Psi(\mathfrak{b})$ \\ \midrule	
$\mathfrak{f}_{4(4)}$ & \begin{xy}
	*\cir<2pt>{}        ="A",
	(10,0) *\cir<2pt>{} ="B",
	(20,0) *\cir<2pt>{} ="C",
	(30,0) *\cir<2pt>{} ="D",
	\ar@{-} "A";"B"
	\ar@{=>} "B";"C"
	\ar@{-} "C";"D"
\end{xy} &
$\left\{ \begin{xy}
	*++!D{b_1} *\cir<2pt>{}        ="A",
	(10,0) *++!D{b_2} *\cir<2pt>{} ="B",
	(20,0) *++!D{b_3} *\cir<2pt>{} ="C",
	(30,0) *++!D{b_4} *\cir<2pt>{} ="D",
	\ar@{-} "A";"B"
	\ar@{=>} "B";"C"
	\ar@{-} "C";"D"
\end{xy} 
\right\}$
\\ \hline
	
$\mathfrak{f}_{4(-20)}$ & \begin{xy}
	*{\bullet}        ="A",
	(10,0) *{\bullet} ="B",
	(20,0) *{\bullet} ="C",
	(30,0) *\cir<2pt>{} ="D",
	\ar@{-} "A";"B"
	\ar@{=>} "B";"C"
	\ar@{-} "C";"D"
\end{xy} &
$\left\{ \begin{xy}
	*++!D{0} *\cir<2pt>{}        ="A",
	(10,0) *++!D{0} *\cir<2pt>{} ="B",
	(20,0) *++!D{0} *\cir<2pt>{} ="C",
	(30,0) *++!D{b} *\cir<2pt>{} ="D",
	\ar@{-} "A";"B"
	\ar@{=>} "B";"C"
	\ar@{-} "C";"D"
\end{xy} 
\right\}$\\
\bottomrule
	\end{longtable}
\end{center}

Therefore for each $\mathfrak{g}$,
we can find a basis of $\Psi(\mathfrak{b})$ 
by taking some weighted Dynkin diagrams 
in $\Psi(\mathcal{H}^n(\mathfrak{a}_+)) = 
\{ \Psi_A \in \Psi(\mathcal{H}^n(\mathfrak{j}_+)) \mid \text{$\Psi_A$ matches $S_\mathfrak{g}$}\}$ as follows:

\begin{center}
\begin{longtable}{ll} \toprule
$\mathfrak{g}$ & Example of basis of $\Psi(\mathfrak{b})$ \\ \midrule
$\mathfrak{f}_{4(4)}$ & $A_2,\tilde{A_2},B_3,F_4$ \\ \hline
$\mathfrak{f}_{4(-20)}$ & $\tilde{A_2}$ \\ 
\bottomrule
\end{longtable}
\end{center}

\subsection{Type $G_2$}

Let us consider the case where 
$\mathfrak{g}_\mathbb{C}$ is of type $G_2$,
that is, 
$\mathfrak{g}_\mathbb{C} \simeq \mathfrak{g}_{2,\mathbb{C}}$.
Then we have 
\[
\Map(\Pi,\mathbb{R})^\iota = \Map(\Pi,\mathbb{R})
= \left\{ \begin{xy}
	*++!D{a_1} *\cir<2pt>{}        ="A",
	(10,0) *++!D{a_2}  *\cir<2pt>{} ="B",
	\ar@3{->} "A";"B"
\end{xy} 
\right\}
\]

In \cite[Section 8.4]{Collingwood-McGovern93}, 
we can find some examples of weighted Dynkin diagrams in $\mathcal{H}^n(\mathfrak{j}_+)$ as follows:

\begin{center}
\begin{longtable}{ll} \toprule
	Symbol & Weighted Dynkin diagram in $\Psi(\mathcal{H}^n(\mathfrak{j}_+))$ \\ \midrule
	$G_2(a_1)$ & \begin{xy}
	*++!D{2} *\cir<2pt>{}        ="A",
	(10,0) *++!D{0}  *\cir<2pt>{} ="B",
	\ar@3{->} "A";"B"
\end{xy}\\ 
	$G_2$ & \begin{xy}
	*++!D{2} *\cir<2pt>{}        ="A",
	(10,0) *++!D{2}  *\cir<2pt>{} ="B",
	\ar@3{->} "A";"B"
\end{xy}\\ 
\bottomrule
\end{longtable}
\end{center}

Let $\mathfrak{g}$ be a non-compact real form of $\mathfrak{g}_\mathbb{C}$.
Then $\mathfrak{g}$ is a split real form of $\mathfrak{g}_\mathbb{C}$.
In particular, the Satake diagram $S_\mathfrak{g}$ and 
$\Psi(\mathfrak{b})$ of $\mathfrak{g}$ are the following:
\begin{center}
	\begin{longtable}{lll} \toprule
	$\mathfrak{g}$ & $S_\mathfrak{g}$ & $\Psi(\mathfrak{b})$ \\ \midrule
	
$\mathfrak{g}_{2(2)}$ & \begin{xy}
	*\cir<2pt>{}        ="A",
	(10,0) *\cir<2pt>{} ="B",
	\ar@3{->} "A";"B"
\end{xy} & 
$\left\{ \begin{xy}
	*++!D{b_1} *\cir<2pt>{}        ="A",
	(10,0) *++!D{b_2}  *\cir<2pt>{} ="B",
	\ar@3{->} "A";"B"
\end{xy} 
\right\}$
\\ 
\bottomrule
	\end{longtable}
\end{center}

Therefore, the weighted Dynkin diagrams $G_2(a_1)$ and $G_2$ 
in $\Psi(\mathcal{H}^n(\mathfrak{a}_+)) = 
\{ \Psi_A \in \Psi(\mathcal{H}^n(\mathfrak{j}_+)) \mid \text{$\Psi_A$ matches $S_\mathfrak{g}$}\}$  
give a basis of $\Psi(\mathfrak{b})$.

\appendix

\section{Semisimple symmetric spaces with proper $SL(2,\R)$-actions}\label{sec:app}

In \cite[Theorem 1.3]{Okuda13cls}, 
by using the main result (Theorem \ref{thm:intro}) of this paper, Kobayashi's properness criterion \cite[Theorem 4.1]{Kobayashi89} and Benoist's results in \cite{Benoist96}, 
we proved that the following three conditions on a semisimple symmetric space $G/H$ are equivalent:
\begin{enumerate}
\item $G/H$ admits a proper action of $SL(2,\R)$ via $G$.
\item $G/H$ admits discontinuous groups which are not virtually abelian.
\item \label{item:JDG:nilp} There exists a (complex) nilpotent orbit $\mathcal{O}^\C$ in $\mathfrak{g}_\C := \mathfrak{g} + \sqrt{-1} \mathfrak{g}$ such that $\mathcal{O}^\C$ meets $\mathfrak{g}$ but does not meet the $c$-dual $\mathfrak{h} + \sqrt{-1}\mathfrak{q}$ of the symmetric pair $(\mathfrak{g},\mathfrak{h})$, 
where $\mathfrak{g} = \mathfrak{h} + \mathfrak{q}$ is the decomposition of $\mathfrak{g}$ with respect to the involution associated to the symmetric pair $(\mathfrak{g},\mathfrak{h})$.
\end{enumerate}

In \cite[\text{Table }6]{Okuda13cls}, 
we gave a list of symmetric pair $(\mathfrak{g},\mathfrak{h})$ with simple $\mathfrak{g}$ satisfying the condition \eqref{item:JDG:nilp}.
Taking this opportunity, we would like to correct 
some errors of $(\mathfrak{g},\mathfrak{h})$ in \cite[\text{Table }6]{Okuda13cls}:
\begin{itemize}
\item ``$(\mathfrak{sl}(n,\R), \mathfrak{so}(n-i,i))$ for $2i < n$'' should be ``$(\mathfrak{sl}(n,\R), \mathfrak{so}(n-i,i))$ for $2i < n-1$''.
\item ``$(\mathfrak{su}(2m-1,2m-1),\mathfrak{so}^*(4m-2))$'' should be ``$(\mathfrak{su}(n,n),\mathfrak{so}^*(2n))$''.
\item ``$(\mathfrak{so}(k,k), \mathfrak{so}(2k,\C) + \mathfrak{so}(2))$'' should be ``$(\mathfrak{so}(k,k), \mathfrak{so}(k,\C) + \mathfrak{so}(2))$''.
\item ``$(\mathfrak{sl}(n,\C), \mathfrak{so}(n-i,i))$ for $2i < n$'' should be ``$(\mathfrak{sl}(n,\C), \mathfrak{so}(n-i,i))$ for $2i < n-1$''.
\end{itemize} 
For the reader's convenience we give a list of 
symmetric pair $(\mathfrak{g},\mathfrak{h})$ with simple $\mathfrak{g}$ satisfying the condition \eqref{item:JDG:nilp} as Table \ref{table:sl2-proper} below.

\begin{center}
\begin{longtable}{ll} \toprule
$\mathfrak{g}$ & $\mathfrak{h}$  \\ \midrule
$\mathfrak{sl}(2k,\mathbb{R})$ & $\mathfrak{sl}(k,\mathbb{C}) \oplus \mathfrak{so}(2)$ \\ \hline
$\mathfrak{sl}(n,\mathbb{R})$ & $\mathfrak{so}(n-i,i)$  \\ 
& $(2i < n-1)$ \\ \hline
$\mathfrak{su}^*(2k)$ & $\mathfrak{sp}(k-i,i)$ \\
& $(2i < k-1)$  \\ \hline
$\mathfrak{su}(2p,2q)$ & $\mathfrak{sp}(p,q)$ \\ \hline
$\mathfrak{su}(n,n)$ & $\mathfrak{so}^*(2n)$ \\ \hline 
$\mathfrak{su}(p,q)$ & $\mathfrak{su}(i,j) \oplus \mathfrak{su}(p-i,q-j) \oplus \mathfrak{so}(2)$ \\ 
&$(\min \{ p,q \} > \min \{ i,j \} + \min \{ p-i,q-j \} )$  \\ \hline
$\mathfrak{so}(p,q)$ & $\mathfrak{so}(i,j) \oplus \mathfrak{so}(p-i,q-j)$ \\ 
$(p+q \text{ is odd})$ & ($\min \{ p,q \} > \min \{ i,j \} + \min \{ p-i,q-j \}$) \\ \hline 
$\mathfrak{sp}(n,\mathbb{R})$ & $\mathfrak{su}(n-i,i) \oplus \mathfrak{so}(2)$ \\ \hline
$\mathfrak{sp}(2k,\mathbb{R})$ & $\mathfrak{sp}(k,\mathbb{C})$ \\ \hline 
$\mathfrak{sp}(p,q)$ & $\mathfrak{sp}(i,j) \oplus \mathfrak{sp}(p-i,q-j)$ \\ 
& $(\min \{ p,q \} > \min \{ i,j \} + \min \{ p-i,q-j \} )$  \\ \hline
$\mathfrak{so}(p,q)$ & $\mathfrak{so}(i,j) \oplus \mathfrak{so}(p-i,q-j)$ \\ 
$(p+q \text{ is even})$ & ($\min \{ p,q \} > \min \{ i,j \} + \min \{ p-i,q-j \}$, \\
 & unless $p=q=2m+1$ and $|i-j| =1$)  \\ \hline
$\mathfrak{so}(2p,2q)$ & $\mathfrak{su}(p,q) \oplus \mathfrak{so}(2)$ \\ \hline
$\mathfrak{so}^*(2k)$ & $\mathfrak{su}(k-i,i) \oplus \mathfrak{so}(2)$ \\ 
& $(2i < k-1)$  \\ \hline
$\mathfrak{so}(k,k)$ & $\mathfrak{so}(k,\mathbb{C}) \oplus \mathfrak{so}(2)$ \\ \hline
$\mathfrak{so}^*(4m)$ & $\mathfrak{so}^*(4m-4i+2) \oplus \mathfrak{so}^*(4i-2)$ \\ \hline
$\mathfrak{e}_{6(6)}$ & $\mathfrak{sp}(2,2)$ \\ \hline
$\mathfrak{e}_{6(6)}$ & $\mathfrak{su}^*(6) \oplus \mathfrak{su}(2)$ \\ \hline
$\mathfrak{e}_{6(2)}$ & $\mathfrak{so}^*(10) \oplus \mathfrak{so}(2)$ \\ \hline
$\mathfrak{e}_{6(2)}$ & $\mathfrak{su}(4,2) \oplus \mathfrak{su}(2)$ \\ \hline
$\mathfrak{e}_{6(2)}$ & $\mathfrak{sp}(3,1)$ \\ \hline
$\mathfrak{e}_{6(-14)}$ & $\mathfrak{f}_{4(-20)}$ \\ \hline
$\mathfrak{e}_{7(7)}$ & $\mathfrak{e}_{6(2)} \oplus \mathfrak{so}(2)$ \\ \hline
$\mathfrak{e}_{7(7)}$ & $\mathfrak{su}(4,4)$ \\ \hline
$\mathfrak{e}_{7(7)}$ & $\mathfrak{so}^*(12) \oplus \mathfrak{su}(2)$ \\ \hline
$\mathfrak{e}_{7(7)}$ & $\mathfrak{su}^*(8)$ \\ \hline
$\mathfrak{e}_{7(-5)}$ & $\mathfrak{e}_{6(-14)} \oplus \mathfrak{so}(2)$ \\ \hline
$\mathfrak{e}_{7(-5)}$ & $\mathfrak{su}(6,2)$ \\ \hline
$\mathfrak{e}_{7(-25)}$ & $\mathfrak{e}_{6(-14)} \oplus \mathfrak{so}(2)$ \\ \hline
$\mathfrak{e}_{7(-25)}$ & $\mathfrak{su}(6,2)$ \\ \hline
$\mathfrak{e}_{8(8)}$ & $\mathfrak{e}_{7(-5)} \oplus \mathfrak{su}(2)$ \\ \hline
$\mathfrak{e}_{8(8)}$ & $\mathfrak{so}^*(16)$ \\ \hline
$\mathfrak{f}_{4(4)}$ & $\mathfrak{sp}(2,1) \oplus \mathfrak{su}(2)$ \\ \hline 
$\mathfrak{sl}(2k,\mathbb{C})$ & $\mathfrak{su}^*(2k)$ \\ \hline 
$\mathfrak{sl}(n,\mathbb{C})$ & $\mathfrak{su}(n-i,i)$ \\ 
& ($2i < n-1$) \\ \hline 
$\mathfrak{so}(2k+1,\mathbb{C})$ & $\mathfrak{so}(2k+1-i,i)$ \\
& ($i < k$) \\ \hline
$\mathfrak{sp}(n,\mathbb{C})$ & $\mathfrak{sp}(n-i,i)$ \\ \hline
$\mathfrak{so}(2k,\mathbb{C})$ & $\mathfrak{so}(2k-i,i)$ \\
& ($i < k$ unless $k=i+1=2m+1$) \\ \hline
$\mathfrak{so}(4m,\mathbb{C})$ & $\mathfrak{so}(4m-2i+1,\mathbb{C}) \oplus \mathfrak{so}(2i-1,\mathbb{C})$ \\ \hline
$\mathfrak{so}(2k,\mathbb{C})$ & $\mathfrak{so}^*(2k)$ \\ \hline
$\mathfrak{e}_{6,\mathbb{C}}$ & $\mathfrak{e}_{6(-14)}$ \\ \hline
$\mathfrak{e}_{6,\mathbb{C}}$ & $\mathfrak{e}_{6(-26)}$ \\ \hline
$\mathfrak{e}_{7,\mathbb{C}}$ & $\mathfrak{e}_{7(-5)}$ \\ \hline
$\mathfrak{e}_{7,\mathbb{C}}$ & $\mathfrak{e}_{7(-25)}$ \\ \hline
$\mathfrak{e}_{8,\mathbb{C}}$ & $\mathfrak{e}_{8(-24)}$ \\ \hline
$\mathfrak{f}_{4,\mathbb{C}}$ & $\mathfrak{f}_{4(-20)}$ \\ 
\bottomrule
\caption{Classification of $(\mathfrak{g},\mathfrak{h})$ satisfying the condition ($\star$)}
\label{table:sl2-proper}
\end{longtable}
\end{center}

Here $k \geq 1$, $m \geq 1$, $n \geq 2$, $p, q \geq 1$ and $i, j \geq 0$.
It should be remarked that $\mathfrak{so}(p,q)$ is simple if and only if $p+q \geq 3$ and $(p,q) \neq (2,2)$, 
$\mathfrak{so}(2k,\mathbb{C})$ is simple if and only if $k \geq 3$.

\section*{Acknowledgements.}
The author would like to give warm thanks to Toshiyuki Kobayashi
and Hiroyuki Ochiai whose comments were of inestimable value for this paper.
The author also thanks to Yosuke Morita and Koichi Tojo for pointing out 
some errors in \cite[\text{Table }6]{Okuda13cls}.
This work is supported by JSPS KAKENHI Grant Number JP16K17594.

\bibliographystyle{amsplain}

\def\cprime{$'$}
\providecommand{\bysame}{\leavevmode\hbox to3em{\hrulefill}\thinspace}
\providecommand{\MR}{\relax\ifhmode\unskip\space\fi MR }
\providecommand{\MRhref}[2]{%
  \href{http://www.ams.org/mathscinet-getitem?mr=#1}{#2}
}
\providecommand{\href}[2]{#2}

\end{document}